\documentclass[pdflatex]{sn-jnl-ohne}

\usepackage{nicefrac,blindtext,tabu,mathtools,color,colonequals,mathtools,stmaryrd,etoolbox}
\usepackage[autostyle]{csquotes}
\usepackage[english]{babel}
\usepackage[normalem]{ulem}

\newcommand{\MIN}{ \textsc{MIN}}
\newcommand{\MAX}{\textsc{MAX}}
\renewcommand{\R}{\mathbb{R}}

\newcommand{\ulambda}{\underline{\lambda}}
\newcommand{\olambda}{\bar{\lambda}}
\newcommand{\Rpg}{\mathbb{R}^{p}_{>}}

\jyear{2022}%

\theoremstyle{thmstyleone}%
\newtheorem{theorem}{Theorem}[section]
\newtheorem{proposition}[theorem]{Proposition}%
\newtheorem{lemma}[theorem]{Lemma}
\newtheorem{corollary}[theorem]{Corollary}

\theoremstyle{thmstyletwo}%
\newtheorem{example}[theorem]{Example}%
\newtheorem{remark}[theorem]{Remark}%

\theoremstyle{thmstylethree}%
\newtheorem{definition}[theorem]{Definition}%
\newtheorem{assumption}[theorem]{Assumption}

\usepackage{natbib}
\setcitestyle{authoryear,open={(},close={)},aysep={}} 

\begin{document}

\title[Using Scalarizations for the Approximation of Multiobjective Problems]{Using Scalarizations for the Approximation of Multiobjective Optimization Problems: Towards a General Theory}


\author*[1]{\fnm{Stephan} \sur{Helfrich}}\email{helfrich@mathematik.uni-kl.de}

\author[1,2]{\fnm{Arne} \sur{Herzel}}
\equalcont{These authors contributed equally to this work.}

\author[1]{\fnm{Stefan} \sur{Ruzika}}\email{ruzika@mathematik.uni-kl.de}
\equalcont{These authors contributed equally to this work.}

\author[2,3]{\fnm{Clemens} \sur{Thielen}}\email{clemens.thielen@tum.de}
\equalcont{These authors contributed equally to this work.}

\affil*[1]{\orgdiv{Department of Mathematics}, \orgname{RPTU Kaiserslautern-Landau}, \orgaddress{\street{Paul-Ehrlich-Str.~14}, \postcode{67663} \city{Kaiserslautern}, \country{Germany}}}

\affil[2]{\orgdiv{TUM Campus Straubing for Biotechnology and Sustainability}, \orgname{Weihenstephan-Triesdorf University of Applied Sciences}, \orgaddress{\street{Am~Essigberg~3}, \postcode{94315} \city{Straubing}, \country{Germany}}}

\affil[3]{\orgdiv{Department of Mathematics}, \orgname{Technical University of Munich}, \orgaddress{\street{Boltzmannstr.~3}, \postcode{85748} \city{Garching}, \country{Germany}}}

\abstract{We study the approximation of general multiobjective optimization problems with the help of scalarizations. 
    Existing results state that multiobjective minimization problems can be approximated well by norm-based scalarizations. However, for multiobjective maximization problems, only impossibility results are known so far. Countering this, we show that all multiobjective optimization problems can, in principle, be approximated equally well by scalarizations. In this context, we introduce a transformation theory for scalarizations that establishes the following: Suppose there exists a scalarization that yields an approximation of a certain quality for arbitrary instances of multiobjective optimization problems with a given decomposition specifying which objective functions are to be minimized / maximized. Then, for each other decomposition, our transformation yields another scalarization that yields the same approximation quality for arbitrary instances of problems with this other decomposition. In this sense, the existing results about the approximation via scalarizations for minimization problems carry over to any other objective decomposition -- in particular, to maximization problems -- when suitably adapting the employed scalarization.
    
    We further provide necessary and sufficient conditions on a scalarization such that its optimal solutions achieve a constant approximation quality. We give an upper bound on the best achievable approximation quality that applies to general scalarizations and is tight for the majority of norm-based scalarizations applied in the context of multiobjective optimization. As a consequence, none of these norm-based scalarizations can induce approximation sets for optimization problems with maximization objectives, which unifies and generalizes the existing impossibility results concerning the approximation of maximization problems.
 
    }

\keywords{Multiobjective Optimization, Approximation, Scalarizations, Norm-based Scalarizations}



\maketitle

\section{Introduction}\label{sec:Introduction}
Multiobjective optimization covers methods and techniques for solving optimization problems with several equally important but conflicting objectives, a field of study which is of growing interest both in theory and real-world applications. In such problems, solutions that optimize all objectives simultaneously usually do not exist. Hence, the notion of \emph{optimality} needs to be refined: A solution is said to be \emph{efficient} if any other solution that is better in some objective is necessarily worse in at least one other objective. The image of an efficient solution under the objectives is called a \emph{nondominated image}. A solution is said to be \emph{weakly efficient} if no other solution exists that is strictly better in each objective. It is widely accepted that some entity, called the decision maker, chooses a final preferred solution among the set of (weakly) efficient solutions.
When no prior preference information is available, a main goal of multiobjective optimization is to compute all nondominated images and, for each nondominated image, at least one corresponding efficient solution. 

\medskip

However, multiobjective optimization problems are typically \emph{inherently difficult}: they are hard to solve exactly~\citep{Ehrgott:book,Serafini:Complexity} and, moreover, the cardinalities of the set of nondominated images  may be exponentially large (or even infinite, e.g., for continuous problems), see e.g.~\cite{Boekler:output-sensitive,Ehrgott+Gandibleux:surveyandbibliographyMOCO} and the references therein.
In general, this impedes the applicability of exact solution methods and strongly motivates the \emph{approximation of multiobjective optimization problems} -- a concept to substantially reduce the number of required solutions while still obtaining a provable solution quality. Here, it is sufficient to find a set of (not necessarily efficient) solutions, called an \emph{approximation set}, that, for each possible image, contains a solution whose 
image is in every objective at least as good up to some multiplicative factor.

\medskip

A \emph{scalarization} is a technique to systematically transform a multiobjective optimization problem into  single-objective optimization problems with the help of additional parameters such as weights or reference points.\footnote{Throughout this paper, we define a scalarization as a set of such transformations (see~Definition~\ref{def:scalarizing}). In the literature, it is common that these transformations follow an underlying construction idea (see Section~\ref{sec:specific-scalarizing}). However, we do not explicitly assume this here.} The solutions obtained by solving these single-objective optimization problems are then interpreted in the context of multiobjective optimization (see, e.g.,~\cite{Ehrgott+Wiecek:SkalaSurvey,Jahn:General-Scalarization,Miettinen:scalarizing-functions,Wierzbicki:scalarizing-functions} for overviews on scalarizations). As a consequence, scalarization techniques are a key concept in multiobjective optimization: They often yield (weakly) efficient solutions and they are used as subroutines in algorithms for solving or approximating multiobjective optimization problems.
Unsurprisingly, there exists a vast amount of research concerning both exact and approximate solutions methods that use scalarizations as building blocks, see~\cite{Boekler:output-sensitive,Holzmann:modified-augmented,Klamroth2015,Wierzbicki:scalarizing-functions} and~\cite{Bazgan+etal:power-weighted-sum,Daskalakis+etal:Chord-Algorithm,Diakonikolas+Yannakakis:epsilon-convex,Glasser+etal:multi-hardness-proceedings,Glasser+etal:multi-hardness,Halfmann+etal:general-approx} and the references therein for a small selection.

A widely-known scalarization -- and probably the most simple example -- is the \emph{weighted sum scalarization}, where single-objective optimization problems are obtained by forming weighted sums of the multiple objective functions while keeping the feasible set unchanged. The weighted sum scalarization is frequently used, among others, in approximation methods for multiobjective optimization problems. In fact, it has been shown that optimal solutions of the weighted sum scalarization can be used to obtain approximation sets for each instance of each multiobjective minimization problem~\citep{Bazgan+etal:power-weighted-sum,Glasser+etal:multi-hardness-proceedings,Glasser+etal:multi-hardness,Halfmann+etal:general-approx}. However, these approximation results crucially rely on the assumption that all objectives are to be minimized. In fact, it is known that, for the weighted sum scalarization as well as for every other scalarization studied so far in the context of approximation, even the union of all sets of optimal solutions of the scalarization obtainable for any choice of its parameters does, in general, \emph{not} constitute an approximation set in the case of~\emph{maximization} problems~\citep{Bazgan+etal:power-weighted-sum,Glasser+etal:multi-hardness-proceedings,Glasser+etal:multi-hardness,Halfmann+etal:general-approx,Helfrich+etal:cones}. Consequently, general approximation methods building on the studied scalarizations cannot exist for multiobjective maximization problems.

\medskip

This raises several \emph{fundamental questions}: Are there intrinsic structural differences between minimization and maximization problems with respect to approximation via scalarizations?  Is it, in general, substantially harder or even impossible to construct a scalarization for maximization problems that is as powerful as the weighted sum scalarization is for minimization problems? More precisely, does there exist a scalarization such that, in arbitrary instances of arbitrary maximization problems, optimal solutions of the scalarization constitute an approximation set? Beyond that, can also optimization problems in which both minimization and maximization objectives appear be approximated by means of scalarizations? If yes, what structural properties are necessary in order for scalarizations to be useful concerning the approximation of multiobjective optimization problems in general?
We answer these questions in this paper and study the power of scalarizations for the approximation of multiobjective optimization problems from a general point of view. We focus on scalarizations built by \emph{scalarizing functions} that combine the objective functions of the multiobjective problem by means of strongly or strictly monotone and continuous functions. This captures many important and broadly-applied scalarizations such as the weighted sum scalarization, the weighted max-ordering scalarization, and norm-based scalarizations~\citep{Ehrgott+Wiecek:SkalaSurvey}\textcolor{black}{, but not scalarizations that change the feasible set. However, most important representatives of the latter class such as the budget constraint scalarization, Benson's method, and the elastic constraint method are capable of finding the whole efficient set and, thus, obviously yield approximation sets with approximation quality equal to one (see~\cite{Ehrgott:book,Ehrgott+Wiecek:SkalaSurvey}).} 

\medskip

We develop a \emph{transformation theory for scalarizations} with respect to approximation in the following sense: Suppose there exists a scalarization that yields an approximation of a certain quality for arbitrary instances of multiobjective optimization problems with a given decomposition specifying which objective functions are to be minimized / maximized. Then, for each other decomposition, our transformation yields another scalarization that yields the same approximation quality for arbitrary instances of problems with this other decomposition. We also study necessary and sufficient conditions for a scalarization such that optimal solutions can be used to obtain an approximation set, and determine an upper bound on the best achievable approximation quality. The computation of this upper bound simplifies for so-called weighted scalarizations and, in particular, is tight for the majority of norm-based scalarizations applied so far in the context of multiobjective optimization. As a consequence of this tightness, none of the above norm-based scalarizations can induce approximation sets for arbitrary instances of optimization problems containing maximization objective functions. Hence, this result unifies and generalizes all impossibility results concerning the approximation of maximization problems obtained in~\cite{Bazgan+etal:power-weighted-sum,Glasser+etal:multi-hardness-proceedings,Glasser+etal:multi-hardness,Halfmann+etal:general-approx,Helfrich+etal:cones}.

\subsection{Related Work}\label{sec:related work}
\emph{General approximation methods} seek to work under very weak assumptions and, thus, to be applicable to large classes of multiobjective optimization problems. In contrast, \emph{specific approximation methods} are tailored to problems with a particular structure. We refer to~\cite{Herzel+etal:survey} for an extensive survey on both general and specific approximation methods for multiobjective optimization problems.

Almost all general approximation methods for multiobjective optimization problems build upon the seminal work of~\cite{Papadimitriou+Yannakakis:multicrit-approx}, who show that, for any $\varepsilon>0$, a $(1 + \varepsilon)$-approximation set (i.e., an approximation set with approximation quality~$1 + \varepsilon$ in each objective) of polynomial size is guaranteed to exist in each instance under weak assumptions. Moreover, they prove that a $(1 + \varepsilon)$-approximation set can be computed in (fully) polynomial time for every $\varepsilon>0$ if and only if the so-called \emph{gap problem}, which is an approximate version of the canonical decision problem associated with the multiobjective problem, can be solved in (fully) polynomial time.

Subsequent work focuses on approximation methods that, given an instance and~$\alpha \geq 1$, compute approximation sets whose cardinality is bounded in terms of the cardinality of the smallest possible $\alpha$-approximation set while maintaining or only slightly worsening the approximation quality~$\alpha$~\citep{Bazgan+etal:min-pareto,Diakonikolas+Yannakakis:approx-pareto-sets,Diakonikolas+Yannakakis:epsilon-convex,Koltun+Papadimitriou:approx-dom-repr,Vassilvitskii+Yannakakis:trade-off-curves}. Additionally, the existence result of~\cite{Papadimitriou+Yannakakis:multicrit-approx} has recently been improved by~\cite{Herzel+etal:dualrestrict}, who show that, for any $\varepsilon>0$, an approximation set that is exact in one objective while ensuring an approximation quality of $1 + \varepsilon$ in all other objectives always exists in each instance under the same assumptions.

As pointed out in~\cite{Halfmann+etal:general-approx}, the gap problem is not solvable in polynomial time unless $\textsf{P}=\textsf{NP}$ for problems whose single-objective version is \textsf{APX}-complete and coincides with the weighted sum problem. For such problems, the algorithmic results of~\cite{Papadimitriou+Yannakakis:multicrit-approx} and succeeding articles cannot be used. Consequently, other works study how the weighted sum scalarization and other scalarizations can be employed for approximation. \cite{Daskalakis+etal:Chord-Algorithm,Diakonikolas+Yannakakis:epsilon-convex} show that, in each instance, a set of solutions such that the convex hull of their images yields an approximation quality can be computed in (fully) polynomial time if and only if there is a (fully) polynomial-time approximation scheme for all single-objective optimization problems obtained via the weighted sum scalarization.

The results of~\cite{Glasser+etal:multi-hardness-proceedings,Glasser+etal:multi-hardness} imply that, in each instance of each $p$-objective \emph{minimization} problem and for any $\varepsilon>0$, a $((1 + \varepsilon) \cdot \delta \cdot p)$-approximation set can be computed in fully polynomial time provided that the objective functions are positive-valued and polynomially computable and a $\delta$-approximation algorithm for the optimization problems induced by the weighted sum scalarization exists. They also give analogous results for more general norm-based scalarizations, where the obtained approximation quality additionally depends on the constants determined by the norm-equivalence between the chosen norm and the $1$-norm.

\cite{Halfmann+etal:general-approx} present a method to obtain, in each instance of each \emph{biobjective minimization} problem and for any $0 < \varepsilon \leq 1$, an approximation set that guarantees an approximation quality of $(\delta \cdot (1 + 2\varepsilon))$ in one objective function while still obtaining an approximation quality of at least $(\delta \cdot (1 + \frac{1}{\varepsilon}))$ in the other objective function, provided a polynomial-time $\delta$-approximation algorithm for the problems induced by the weighted sum scalarization is available. This \enquote{trade-off} between the approximation qualities in the individual objectives is studied in more detail by~\cite{Bazgan+etal:power-weighted-sum}, who introduce a multi-factor notion of approximation and present a method that, in each instance of each $p$-objective \emph{minimization} problem for which a polynomial-time $\delta$-approximation algorithm for the problems induced by the weighted sum scalarization exists, computes a set of solutions such that each feasible solution is component-wise approximated within some (possibly solution-dependent) vector $(\alpha_1, \ldots, \alpha_p)$ of approximation qualities $\alpha_i \geq 1$ such that $\sum_{i: \alpha_i > 1} \alpha_i = \delta \cdot p + \varepsilon$.

From another point of view, the weighted sum scalarization can be interpreted as a special case of ordering relations that use cones to model preferences. \cite{Vanderpooten+etal:covers+approximations} study approximation in the context of general ordering cones and characterize how approximation with respect to some ordering cone carries over to approximation with respect to some larger ordering cone. In a related paper,~\cite{Helfrich+etal:cones} focus on \emph{biobjective minimization problems} and provide structural results on the approximation quality that is achievable with respect to the classical (Pareto) ordering cone by solutions that are efficient or approximately efficient with respect to larger ordering cones.

Notably, none of the methods and approximation results for minimization problems provided in~\cite{Bazgan+etal:power-weighted-sum,Glasser+etal:multi-hardness-proceedings,Glasser+etal:multi-hardness,Halfmann+etal:general-approx,Helfrich+etal:cones} can be translated to maximization problems in general: \cite{Glasser+etal:multi-hardness-proceedings,Glasser+etal:multi-hardness} and \cite{Halfmann+etal:general-approx} show that similar approximation results are impossible to obtain in polynomial time for maximization problems unless $\textsf{P} = \textsf{NP}$. \cite{Bazgan+etal:power-weighted-sum} provide, for any $p\geq2$ and polynomial~$\text{pol}$, an instance~$I$ with encoding length~$\lvert I \rvert$ of a $p$-objective maximization problem such that at least one solution not obtainable as an optimal solution of the weighted sum scalarization is not approximated by solutions that are obtainable in this way within a factor of $2^{\text{pol}(\lvert I \rvert)}$ in $p-1$ of the objective functions. 
Similarly, \cite{Helfrich+etal:cones} show that, for any set~$P$ of efficient solutions with respect to some larger ordering cone and any $\alpha \geq 1$, an instance of a biobjective maximization problem can be constructed such that the set~$P$ is not an $\alpha$-approximation set (in the classical sense).

To the best of our knowledge, the only known results tailored to general maximization problems are presented by~\cite{Bazgan+etal:fixed-number}. Here, rather than building on scalarizations, additional severe structural assumptions on the set of feasible solutions are proposed in order to obtain an approximation.

In summary, most of the known approximation methods that build on scalarizations focus on minimization problems. In fact, mainly impossibility results are known concerning the application of such methods for maximization problems and, to the best of our knowledge, a scalarization-based approximation of optimization problems with both minimization and maximization objectives has so far not been considered at all.

\subsection{Our Contribution}\label{sec:our contribution}
We study the power of optimal solutions of scalarizations with respect to approximation. We focus on scalarizations built by \emph{scalarizing functions} that combine the objective functions of the multiobjective problem by means of strongly or strictly monotone and continuous functions. 
In particular, we address the questions outlined above and study why existing approximation results and methods using scalarizations typically work well for minimization problems, but do not yield any approximation quality for maximization problems in general. To this end, we develop a transformation theory for scalarizations with respect to approximation in the following sense: Suppose there exists a scalarization that yields an approximation of a certain quality for arbitrary instances of multiobjective optimization problems with a given decomposition specifying which objective functions are to be minimized / maximized. Then, for each other decomposition, our transformation yields another scalarization that yields the same approximation quality for arbitrary instances of problems with this other decomposition. Hence, our results show that, in principle, the decomposition of the objectives into minimization and maximization objectives does not have an impact on how well multiobjective problems can be approximated via scalarizations. In particular, this shows that, with respect to approximation, equally powerful scalarizations exist for (pure) minimization and (pure) maximization problems and any other possible decomposition of the objectives into minimization and maximization objectives. Consequently, the lack of positive approximation results for maximization problems in the literature is not based on a general impossibility. Rather, it results from the fact that the scalarizations that work well for minimization problems (such as the weighted sum scalarization) have been used also for maximization problems, while our results show that different scalarizations work for the maximization case.

We further provide necessary and sufficient conditions for a scalarization such that optimal solutions of the scalarization can be used to obtain approximation sets for arbitrary instances of multiobjective problems with a certain objective decomposition. We give an upper bound on the best achievable approximation quality solely depending on the level sets of the scalarizing functions contained in the scalarization. We show that the computation of this upper bound simplifies for weighted scalarizations, and provide classes of scalarizations, which include all norm-based scalarizations applied in the context of multiobjective optimization, for which this upper bound is in fact tight. As a consequence of this tightness, none of the above norm-based scalarizations can induce approximation sets for arbitrary instances of optimization problems containing maximization objectives. Hence, this result unifies and generalizes all impossibility results concerning the approximation of maximization problems obtained in~\cite{Bazgan+etal:power-weighted-sum,Glasser+etal:multi-hardness-proceedings,Glasser+etal:multi-hardness,Halfmann+etal:general-approx,Helfrich+etal:cones}.


\section{Preliminaries}\label{sec:Preliminaries}
In this section, we revisit basic concepts from multiobjective optimization and state the assumptions made in this article. For a thorough introduction to the field of multiobjective optimization, we refer to~\cite{Ehrgott:book}.
In the following, if, for a set $Y \in \R^p$ and some index~$i \in \{1, \ldots,p\}$, there exists a $q \in \R^p$ such that $y_i \geq q_i$ for all $y \in Y$, we say that \emph{$Y$ is bounded from below in $i$ (by $q$)}. If there exists a $q \in \R^p$ such that $y_i \leq q_i$ for all $y \in Y$, we say that \emph{$Y$ is bounded from above in $i$ (by $q$)}. Note that a set $Y \subseteq \R^p$ is bounded (in the classical sense) if and only if $Y$ is bounded from above in all $i$ and bounded from below in all $i$.

\medskip

\noindent
We consider general multiobjective optimization problems with $p$~objectives each of which is to be minimized or maximized: Let $p \in \mathbb{N} \setminus \{0\}$ be, as is usually the case in multiobjective optimization, a fixed constant, and let 
\textcolor{black}{$\MIN \in 2^{\{1,\ldots,p\}}$, $\MAX \coloneqq \{1,\ldots,p\} \setminus \MIN$, and $\Pi \coloneqq (\MIN,\MAX)$}. 
Then, we call $\Pi$ an \emph{objective decomposition}
and we define multiobjective optimization problems as follows:
\begin{definition}\label{def:optimization problem}
	Let $\Pi = (\MIN,\MAX)$ be an objective decomposition. A \emph{$p$-objective optimization problem of type~$\Pi$} is given by a set of instances. Each instance~$I= (X,f)$ consists of a set~$X$ of feasible solutions and a vector $f = (f_1,\ldots, f_p)$ of objective functions~$f_i\colon X \to \R$, $i=1,\ldots, p$, where the objective functions~$f_i, i \in \MIN$, are to be minimized and the objective functions~$f_i, i \in \MAX$, are to be maximized.
	If $\MIN = \{1,\ldots, p\}$ and $\MAX = \emptyset$, the $p$-objective optimization problem of type~$\Pi$ is called a \emph{$p$-objective minimization problem}.
	If $\MIN =  \emptyset$ and $\MAX = \{1,\ldots, p\}$, the $p$-objective optimization problem of type~$\Pi$ is called a \emph{$p$-objective maximization problem}.
\end{definition}
Component-wise orders on $\mathbb{R}^p$, based on a given objective decomposition, induce relations between images of solutions:  
\begin{definition}\label{def:component-wise-order}
    Let $\Pi = (\MIN,\MAX)$ be an objective decomposition. For $y,y ' \in\mathbb{R}^p$, the \emph{weak component-wise order}, the \emph{component-wise order}, and the \emph{strict component-wise order (with respect to $\Pi$)} are defined by
    \begin{align*}
    &y \leqq_{\Pi} y' \;\vcentcolon\Leftrightarrow\; y_i \leq y'_i\; \text{for all}\; i \in \MIN,\; y_i \geq y'_i\; \text{for all}\; i \in \MAX,\\
    &y \leq_{\Pi} y' \; \vcentcolon\Leftrightarrow\; y_i \leq y'_i\; \text{for all}\; i \in \MIN,\; y_i \geq y'_i\; \text{for all}\; i \in \MAX \; \text{and}\; y \neq y',\\
    &y <_{\Pi} y' \; \vcentcolon\Leftrightarrow\; y_i < y'_i\; \text{for all}\; i \in \MIN,\; y_i > y'_i\; \text{for all}\; i \in \MAX,
    \end{align*}
    respectively. Furthermore, we write~$\R^p_{>} \coloneqq \{y \in \R^p \, \vert \, 0 < y_i, i=1,\ldots,p\}$.
\end{definition}\noindent
Based on these component-wise orders, multiobjective notions of optimality can be defined: 
\begin{definition}\label{def:domination}
	Let $\Pi$ be an objective decomposition. In an instance of a $p$-objective optimization problem of type $\Pi$, a solution~$x \in X$ (strictly) \emph{dominates} another solution $x' \in X$ if $f(x) \leq_{\Pi} f(x')$ ($f(x) <_{\Pi} f(x')$).
	A solution~$x \in X$ is called \emph{(weakly) efficient} if there does not exist any solution~$x' \in X$ that (strictly) dominates $x$. If a solution $x \in X$ is (weakly) efficient, then the corresponding point $y = f(x) \in \R^p$ is called \emph{(weakly) nondominated}.
	The set~$X_E \subseteq X$ of efficient solutions is called the \emph{efficient set}. The set~$Y_N= f(X_E) \subseteq \R^p$ of nondominated images is called the \emph{nondominated set}.
\end{definition}
In each instance of a $p$-objective optimization problem, it is then the goal to return a set~$X^* \subseteq X$ of feasible solutions whose image $f(X^*)$ under $f\colon X\to \R^p$ is the nondominated set~$Y_N$.

\medskip

One main issue of a multiobjective optimization problem is that the nondominated set $Y_N$ may consist of exponentially many images in general~\citep{Boekler:output-sensitive,Ehrgott+Gandibleux:surveyandbibliographyMOCO}, i.\,e., such problems are intractable. Approximation is a concept to substantially reduce the number of solutions that must be computed. Instead of requiring at least one corresponding efficient solution for each nondominated image, a solution whose image \enquote{almost} (by means of a multiplicative factor) dominates the nondominated image is sufficient. To ensure that approximation is meaningful and well-defined, a typical assumption made in the literature on both the approximation of single-objective and multiobjective optimization problems (see~\cite{Williamson+Shmoys:Approximation} and \cite{Bazgan+etal:fixed-number,Bazgan+etal:min-pareto,Diakonikolas+Yannakakis:approx-pareto-sets,Papadimitriou+Yannakakis:multicrit-approx,Vanderpooten+etal:covers+approximations}, respectively) is also used in this work:
\begin{assumption}\label{ass:positivity}
    In any instance of each $p$-objective optimization problem, the set $Y = f(X)$ of feasible points is a subset of~$\Rpg$. That is, the objective functions~$f_i \colon X \rightarrow \R_>$ map solutions to positive values, only.
\end{assumption}
\noindent
Approximation is then formally defined as follows:
\begin{definition}\label{def:approximation}
	Let $\Pi = (\MIN,\MAX)$ be an objective decomposition and let $\alpha \geq 1$ \textcolor{black}{be a constant}.
 In an instance~$I = (X,f)$ of a $p$-objective optimization problem of type~$\Pi$, we say that~$x' \in X$ is \emph{$\alpha$-approximated by}~$x \in X$, or~$x$ \emph{$\alpha$-approximates} $x'$, if $f_i(x) \leq \alpha \cdot f_i(x')$ for all $i \in \MIN$ and $f_i(x) \geq \frac{1}{\alpha} \cdot f_i(x')$ for all $i \in \MAX$. 
	A set~$P_\alpha \subseteq X$ of solutions is called an \emph{$\alpha$-approximation set} if, for any feasible solution~$x' \in X$, there exists a solution~$x \in P_\alpha$ that $\alpha$-approximates~$x'$.
\end{definition}

\medskip

\emph{Scalarizations} are common approaches to obtain (efficient) solutions~\citep{Ehrgott:book}. One large class of scalarizations transforms a multiobjective optimization problem into a single-objective optimization problem with the help of \emph{scalarizing functions}:
\begin{definition}\label{def:scalarizing}
	Given an objective decomposition~$\Pi$, a function $s\colon \Rpg \to \R$ is called
	\begin{itemize}
		\item \emph{strongly $\Pi$-monotone} if $y \leq_{\Pi} y'$ for~$y,y' \in \Rpg$ implies $s(y) < s(y')$, and
		\item \emph{strictly $\Pi$-monotone} if $y <_{\Pi} y'$ for $y,y' \in \Rpg$ implies $s(y) < s(y')$.
	\end{itemize}
	Then, a \emph{scalarizing function for~$\Pi$} is a function~$s\colon \Rpg \to \R$ that is continuous and (at least) strictly $\Pi$-monotone. The \emph{level set} of $s$ at some point~$y'$ is denoted by
    \begin{align*}
        L(y',s) \coloneqq \left\{ y \in \Rpg \, \vert \, s(y) = s(y') \right\}.
    \end{align*}
	A set~$S$ of scalarizing functions is referred to as a~\emph{scalarization for $\Pi$}.
\end{definition}
This definition is motivated by norm-based scalarizations~\citep{Ehrgott:book} and captures several important scalarizations such as the weighted sum scalarization (see Example~\ref{ex:weighted_sum}). These scalarizations typically subsume only scalarizing functions that follow the same underlying construction idea. Such a construction is motivated, for example, by (polynomial-time) solvability of the obtained single-objective optimization problems. However, we allow scalarizations to contain various different scalarizing functions for the sake of generality.


With the help of scalarizing functions, any instance of a multiobjective optimization problem can be transformed into instances of a single-objective optimization problem, for which solution methods are widely studied. 
\begin{definition}\label{def:scalarizing-function-optimal-solution}
	Let $s\colon \Rpg \to \R$ be a scalarizing function for an objective decomposition~$\Pi$. 
	In an instance~$I = (X,f)$ of a multiobjective optimization problem of type~$\Pi$, a solution $x \in X$ is called \emph{optimal for~$s$} if $s(f(x)) \leq s(f(x'))$ for each $x' \in X$.  
\end{definition}
Note that the minimization of the instance of a single-objective problem obtained by scalarizing functions (which is implicitly assumed both in Definition~\ref{def:scalarizing} and Definition~\ref{def:scalarizing-function-optimal-solution}) is without loss of generality. One could alternatively define strongly (strictly) $\Pi$-monotonicity of functions $s$ via $y \leq_{\Pi} y'$ ($y <_{\Pi} y'$) implies $s(y) > s(y')$, and optimality for $s$ of a solution $x \in X$ via $s(f(x)) \geq s(f(x'))$ for all $x' \in X$. Then, all results in this work are still valid.

In order to guarantee that optimal solutions exist for any scalarizing function, we additionally assume:
\begin{assumption}\label{ass:compact}
    In any instance of each $p$-objective optimization problem, the set $Y = f(X)$ of feasible points is compact.
\end{assumption}
Note that Assumption~\ref{ass:compact} is satisfied for a large variety of well-known optimization problems, including multiobjective formulations of (integer/mixed integer) linear programs with compact feasible sets, nonlinear problems with continuous objectives and compact feasible sets, and all combinatorial optimization problems. 

Summarizing, we assume that, in any instance of each $p$-objective optimization problem, the set $Y = f(X)$ of feasible points is a compact subset of $\Rpg$. This implies that the objective functions $f_i\colon X \rightarrow \R_>$ map solutions to positive values only, and that the set of images of feasible solutions is guaranteed to be bounded from below in all $i$ (by the origin). Hence, the set of images is bounded if and only if it is bounded from above in all $i$.

\medskip 
Before we interpret scalarizing functions and their optimal solutions in the context of multiobjective optimization, we collect some useful properties.
\begin{lemma}\label{lem:beam}
    Let~$\Pi = (\MIN, \MAX)$ be an objective decomposition. Let $s\colon \Rpg \rightarrow \R$ be a scalarizing function for~$\Pi$. Let $q, y \in \Rpg$. Then, there exists $\lambda \in \R_>$ such that~$s(q') = s(y)$, where $q' \in \Rpg$ is defined by $q'_i \coloneqq \lambda \cdot q_i$ for all $i \in \MIN$, $q'_i \coloneqq \frac{1}{\lambda} \cdot q_i$ for all $i \in \MAX$.
\end{lemma}
\begin{proof}
	Without loss of generality, let $\MIN = \{1,\ldots, k\}$ and $\MAX = \{k+1,\ldots, p\}$ for some $k \in \{0, \ldots, p\}$. Otherwise, the objectives may be reordered accordingly. Consider the function~$s_q \colon \R_> \to \R, s_q(\lambda) \coloneqq s( (\lambda \cdot q_1, \ldots, \lambda \cdot q_k,\frac{1}{\lambda} \cdot q_{k+1},\ldots,\frac{1}{\lambda} \cdot q_p))$. 
	Then,~$s_q$ is a continuous function. Choose \begin{align*}
	\ulambda &\coloneqq \frac{1}{2} \cdot \min\left\{ \frac{y_1}{q_1}, \ldots, \frac{y_k}{q_k}, \frac{q_{k+1}}{y_{k+1}}, \ldots, \frac{q_p}{y_p} \right\} \text{ and } \\ \olambda &\coloneqq 2 \cdot \max\left\{ \frac{y_1}{q_1},\ldots,\frac{y_k}{q_k},  \frac{q_{k+1}}{y_{k+1}},\ldots, \frac{q_p}{y_p}\right\}.
	\end{align*}
	Then~$\ulambda \cdot q_i < y_i < \olambda \cdot q_i$ for all $i =1, \ldots,k$ and $\ulambda \cdot q_i > y_i > \olambda \cdot q_i$ for all $i =k+1,\ldots, p$, which implies that
	\begin{align*}
	s_q(\ulambda) < s(y) < s_q(\olambda). 
	\end{align*}
	Since $s_q$ is continuous, by the intermediate value theorem, there exists some $\lambda \in \R_>$ such that \begin{align*}s\left(\left(\lambda \cdot q_1, \ldots, \lambda \cdot q_k,\frac{1}{\lambda} \cdot q_{k+1},\ldots,\frac{1}{\lambda} \cdot q_p\right)\right) = s_q(\lambda) = s(y).\end{align*}
\end{proof}
\begin{lemma}\label{lem:weak-monotonicity}
    Let $s\colon \Rpg \rightarrow \R$ be a scalarizing function for some objective decomposition~$\Pi$. Let $y, y' \in \Rpg$. Then,~$y \leqq_{\Pi} y'$ implies $s(y) \leq s(y')$.
\end{lemma}
\begin{proof}
    Again, let $\Pi = (\{1,\ldots,k\},\{k+1,\ldots, p\} )$ for some $k \in \{0,\ldots,p\}$ without loss of generality. Let~$y \leqq_{\Pi} y'$. For the sake of a contradiction, assume that $s(y) > s(y')$. Then, by Lemma~\ref{lem:beam}, there exists~$\lambda \in \R_>$ such that
    $s(q) = s(y')$, where $q \in \Rpg$ is defined by
    \begin{align*}
        q \coloneqq \left(\lambda \cdot y_1,\ldots, \lambda \cdot y_k, \frac{1}{\lambda} \cdot y_{k+1},\ldots, \frac{1}{\lambda} \cdot y_p \right).
    \end{align*}
    Note that $\lambda < 1$ since, otherwise, either $q = y$ or $q >_{\Pi} y$ and, thus, $s(y') = s(q) \geq s(y)$ by the strict monotonicity of $s$. We obtain $q <_{\Pi} y \leqq_{\Pi} y'$ and, therefore, $s(q) < s(y')$ contradicting that $s(q) = s(y')$.
\end{proof}

\medskip

Concerning scalarizing functions, a natural question is whether optimal solutions for a scalarizing function~$s$ are always efficient.
\begin{proposition}
	Let $\Pi$ be an objective decomposition. Let~$I = (X,f)$ be an instance of a $p$-objective optimization problem of type $\Pi$ and~$s\colon \Rpg \to \R$ be a scalarizing function for $\Pi$. Then any solution~$x \in X$ that is optimal for~$s$ is weakly efficient. Moreover, there exists a solution~$x \in X$ that is optimal for~$s$ and also efficient. If~$s$ is strongly $\Pi$-monotone, then any solution~$x \in X$ that is optimal for~$s$ is efficient.
\end{proposition}
\begin{proof}
	Let~$x \in X$ be an optimal solution for~$s$. Assume that~$x$ is not weakly efficient, i.e., there exists~$x' \in X$ such that $f(x') <_{\Pi} f(x)$. Then the strict $\Pi$-monotonicity of $s$ implies that $s(f(x')) < s(f(x))$ contradicting that~$x$ is optimal for~$s$. 
	
	Since~$f(X)$ is a compact set, the continuous function~$s$ attains its minimum on~$f(X)$, i.e., there exists a solution~$x \in X$ that is optimal for~$s$. Moreover, it is well-known that the nondominated set~$Y_N$ is externally stable if $Y = f(X)$ is compact~\citep{Ehrgott:book}. Thus, if $x$ is not efficient, there exists an efficient solution~$x'$ dominating~$x$. Lemma~\ref{lem:weak-monotonicity} yields that $s(f(x')) \leq s(f(x))$, which implies that (the efficient solution) $x'$ is also optimal for~$s$.
	
	If~$s$ is strongly $\Pi$-monotone, then, for any solution~$x \in X$ that is not efficient, there exists a solution~$x' \in X$ with~$f(x') \leq_{\Pi} f(x)$ and, therefore, $s(f(x')) < s(f(x))$. Hence, any optimal solution for~$s$ must be efficient.		
\end{proof}
\begin{example}\label{ex:weighted_sum}
    Consider an objective decomposition~$\Pi = (\MIN, \MAX)$. Then, for any (fixed) weight vector~$w = (w_1, \ldots,w_p) \in \Rpg$, the function $s_w \colon \R_>^p \to \R, s_w(y) \coloneqq \sum_{i \in \MIN} w_i \cdot y_i - \sum_{i \in \MAX} w_i \cdot y_i$ defines a scalarizing function that is strongly $\Pi$-monotone. 
    The scalarizing function~$s_w$ is called \emph{weighted sum scalarizing function} with weights~$w_1, \ldots,w_p$. The set of all weighted sum scalarizing functions~$\{ s_w(y) = \sum_{i \in \MIN} w_i \cdot y_i - \sum_{i \in \MAX} w_i \cdot y_i \; \vert \; w \in \Rpg \}$ is called \emph{weighted sum scalarization}.
    Typically, a feasible solution that is optimal for some weighted sum scalarizing function is called \emph{supported}~\citep{Bazgan+etal:power-weighted-sum,Ehrgott:book}. 
\end{example}

\textcolor{black}{Note that, in the literature, the single-objective optimization problem obtained by a weighted sum scalarizing function applied to instances of multiobjective \emph{maximization} problems typically reads as $\max_{x \in X} \sum_{i = 1}^p w_i \cdot f_i(x)$. In our notation used in the above example, the single-objective problem reads as $\min_{x\in X} - \sum_{i = 1}^p w_i \cdot f_i(x)$. Since $\min_{x\in X} - \sum_{i = 1}^p w_i \cdot f_i(x) = - \max_{x \in X} \sum_{i = 1}^p w_i \cdot f_i(x)$, the optimization problems are indeed equivalent in the sense that the optimal solution sets of both problems coincide.}

\medskip

\noindent
We generalize the concept of supportedness to arbitrary scalarizations:
\begin{definition}
	Let $S$ be a scalarization (of finite or infinite cardinality) for an objective decomposition~$\Pi$. In an instance of a multiobjective optimization problem of type~$\Pi$, a solution $x \in X$ is called \emph{$S$-supported} if there exists a scalarizing function $s \in S$ such that~$x$ is optimal for~$s$.
\end{definition}
Note that optimal solutions for a scalarizing function are not necessarily unique. Moreover, different optimal solutions can be mapped to different images in $\R^p$, and, thus, contribute different information to the solution/approximation process. However, given a scalarization~$S$, it is often the case that not all $S$-supported solutions must be computed to draw conclusions on the nondominated set. Instead, it is sufficient to compute a set of solutions that contains, for each~$s \in S$, at least one solution that is optimal for $s$, see, for example,~\cite{Boekler:output-sensitive}. Hence, we define:
\begin{definition}
    Let $S$ be a scalarization (of finite or infinte cardinality) for an objective decomposition~$\Pi$. In an instance of a multiobjective optimization problem of type~$\Pi$, a set of solutions~$P \subseteq X$ is an \emph{optimal solution set for $S$}, if, for each scalarizing function~$s \in S$, there is a solution~$x \in P$ that is optimal for $s$.
\end{definition}
Note that the set of $S$-supported solutions is the largest optimal solution set for $S$ in the sense that it is the union of all optimal solution sets for $S$.
\begin{example}\label{ex:weighted-sum-optimal-solution-set}
    Let $I = (X,f)$ be an instance of a bi-objective minimization problem such that $Y = \text{conv}(\{q^1,q^2\})$ for some $q^1,q^2 \in \Rpg$ with $q^1_1 < q^2_1$ and $q^1_2 > q^2_2$. Let~$S$ be the weighted sum scalarization (see Example~\ref{ex:weighted_sum}).  Then, $X$ is the set of ($S$-) supported solutions, and $\{x^1,x^2\}$ with $f(x^1) = q^1$ and $f(x^2) = q^2$ is an optimal solution set for~$S$ with minimum cardinality.
\end{example}

\section{Transforming Scalarizations}\label{sec:duality}
In this section, we study the approximation quality that can be achieved for multiobjective optimization problems by means of optimal solutions of scalarizations.  
Countering the existing impossibility results for maximization problems (see Section~\ref{sec:related work}), we show that, in principle, scalarizations may serve as building blocks for the approximation of any multiobjective optimization problem: If there exists a scalarization~$S$ for an objective decomposition~$\Pi$ such that, in each instance of each multiobjective optimization problem of type~$\Pi$, every optimal solution set for~$S$ is an approximation set, then, for any other objective decomposition~$\Pi'$, there exists a scalarization~$S'$ for which the same holds true (with the same approximation quality).
To this end, given a set $\Gamma \subseteq \{1,\ldots,p\}$, we define a \enquote{flip function}~$\sigma^{\Gamma}\colon\Rpg \rightarrow \Rpg$ via
\begin{align}\label{eq:flipper}
    \sigma^{\Gamma}_i(y) &\coloneqq \begin{cases} \frac{1}{y_i}, &\text{ if } i \in \Gamma,\\
    y_i,&\text{ else.}
    \end{cases}
\end{align}
Note that $\sigma^{\Gamma}$ is continuous, bijective, and self-inverse, i.e., $\sigma^{\Gamma}(\sigma^{\Gamma}(y)) = y$ for all $y \in \Rpg$. 

\medskip 

In the remainder of this section, let an objective decomposition~$\Pi = (\MIN,\MAX)$ be given. Using~$\sigma^{\Gamma}$, we define a transformed objective decomposition by reversing the direction of optimization of all objective functions $f_i$, $i \in \Gamma$. Formally, this is done as follows: 
\begin{definition}\label{def:dual-objective-decomposition}
    For $\Pi = (\MIN,\MAX)$, the \emph{$\Gamma$-transformed decomposition} $\Pi^{\Gamma} = (\MIN^{\Gamma},\MAX^{\Gamma})$ (of $\Pi$) is defined by $\MIN^{\Gamma} \coloneqq (\MIN \setminus \Gamma) \cup (\Gamma \cap \MAX)$ and $\MAX^{\Gamma} \coloneqq (\MAX \setminus \Gamma) \cup (\Gamma \cap \MIN)$.
\end{definition} 
It is known (e.g., from~\cite{Papadimitriou+Yannakakis:multicrit-approx}) that any $p$-objective optimization problem of type~$\Pi$ can be transformed with the help of $\sigma^{\Gamma}$ to a $p$-objective optimization problem of type~$\Pi^{\Gamma}$: for any instance $I = (X,f)$ of a given $p$-objective optimization problem of type~$\Pi$, define an instance~$I^{\Gamma} =(X^{\Gamma},f^{\Gamma})$ of some $p$-objective optimization problem of type~$\Pi^{\Gamma}$ via $X^{\Gamma} \coloneqq X$ and  $f^{\Gamma}\colon X \rightarrow \Rpg, f^{\Gamma}(x) \coloneqq \sigma^{\Gamma}(f(x))$. 
The instance~$I^{\Gamma}$ is equivalent to~$I$ in the sense that, for any two solutions~$x, x^{\Gamma} \in X$, the solution~$x^{\Gamma}$ is (strictly) dominated by the solution~$x$ in $I$ if and only if $x^{\Gamma}$ is (strictly) dominated by~$x$ in~$I^{\Gamma}$. \textcolor{black}{Moreover, it is easy to see that our assumption that the set of feasible points is a compact subset of $\Rpg$ is preserved under this transformation: If $f(X)$ is a compact subset of $\Rpg$, then $f^{\Gamma}(X)$ is also a compact subset of $\Rpg$.} Further, this transformation is compatible with the notion of approximation: For any $\alpha \geq 1$, a solution~$x' \in X$ is $\alpha$-approximated by a solution $x \in X$ in~$I$ if and only if~$x'$ is $\alpha$-approximated by~$x$ in~$I^{\Gamma}$. This means that a set $P_\alpha \subseteq X$ of solutions is an $\alpha$-approximation set for~$I$ if and only if $P_\alpha$ is an $\alpha$-approximation set for~$I^{\Gamma}$. 
Note that this transformation is self-inverse, i.e., $(I^{\Gamma})^{\Gamma}=I$. Thus, we call $I^{\Gamma}$ the \emph{$\Gamma$-transformed instance} of~$I$, and the $p$-objective optimization problem of type~$\Pi^{\Gamma}$ that consists of all $\Gamma$-transformed instances the \emph{$\Gamma$-transformed optimization problem}.
Similarly, we define $\Gamma$-transformed scalarizing functions:
\begin{definition}\label{def:transformed-scalarization}
    Let $s\colon \Rpg \to \R$ be a scalarizing function for $\Pi$ and let $\Gamma \subseteq \{1,\ldots, p\}$. We define the \emph{$\Gamma$-transformed scalarizing function}~$s^{\Gamma}\colon \Rpg \to \R$ (of $s$) by
    \begin{align*}
        s^{\Gamma}(y) \coloneqq s\left(\sigma^{\Gamma}(y) \right).
    \end{align*}
    Given a scalarization~$S$ for $\Pi$, we call $S^{\Gamma} \coloneqq \left\{s^{\Gamma} \, \big\vert \, s \in S\right\}$ the \emph{$\Gamma$-transformed scalarization} (of $S$).
\end{definition}\noindent
Note that the scalarizing function~$\left(s^{\Gamma}\right)^{\Gamma} $, i.e., the $\Gamma$-transformed scalarizing function of the $\Gamma$-transformed scalarizing function of~$s$, equals the scalarizing function~$s$: For each $y \in \Rpg$, we have
\begin{align*}
    \left(s^{\Gamma}\right)^{\Gamma}(y) = s^{\Gamma}\left(\sigma^{\Gamma}(y)\right) =  s\left(\sigma^{\Gamma}\left(\sigma^{\Gamma}(y)\right) \right) = s(y).
\end{align*}
The next lemma shows that scalarizing functions for $\Pi$ are indeed mapped to scalarizing functions for~$\Pi^{\Gamma}$:
\begin{lemma}\label{lem:transform-is-scalarizing}
    Let $s$ be a scalarizing function for $\Pi$. Then, $s^{\Gamma}$ is a scalarizing function for $\Pi^{\Gamma}$. 
\end{lemma}
\begin{proof}
    Since $s$ and $\sigma^{\Gamma}$ are continuous, $s^{\Gamma}$ is continuous as well. Let $y,y' \in \Rpg$ such that $y <_{\Pi^{\Gamma}} y'$. Then, $\sigma^{\Gamma}(y) <_{\Pi} \sigma^{\Gamma}(y')$ and, since $s$ is strictly $\Pi$-monotone, $s(\sigma^{\Gamma}(y)) \leq s(\sigma^{\Gamma}(y'))$. That is, the function~$s^{\Gamma}$ is strictly $\Pi^{\Gamma}$-monotone. 
\end{proof}
As discussed in Remark~\ref{rem:dual-alternatives} below, several meaningful, but not self-inverse, definitions of a $\Gamma$-transformed scalarizing function~$s^{\Gamma}$ exist. 

The next lemma shows that the $\Gamma$-transformed scalarizing function~$s^{\Gamma}$ of a scalarizing function~$s$ preserves optimality of a solution~$x$ in the sense that $x$ is optimal for $s$ in $I$ if and only if $x$ is optimal for $s^{\Gamma}$ in~$I^{\Gamma}$.
\begin{lemma}\label{lem:transformed-scalarization-optimal}
	Let $I = (X,f)$ be an instance of a $p$-objective optimization problem of type~$\Pi$ and let~$s\colon \Rpg \to \R$ be a scalarizing function for $\Pi$. Then a solution~$x\in X$ is an optimal solution for~$s$ in $I$ if and only if~$x$ is an optimal solution for $s^{\Gamma}$ in~$I^{\Gamma}$.
\end{lemma}
\begin{proof}  
    Note that
    \begin{align*} 
    s^{\Gamma}(\sigma^{\Gamma}(f(x''))) = s(\sigma^{\Gamma}(\sigma^{\Gamma}(f(x''))) = s(f(x''))
    \end{align*} 
    for all $x'' \in X$. This implies, for any~$x,x' \in X$, that
    $s(f(x)) \leq s(f(x'))$ if and only if $s^{\Gamma}(\sigma^{\Gamma}(f(x))) \leq s^{\Gamma}(\sigma^{\Gamma}(f(x')))$ and, hence, a feasible solution $x$ is optimal for $s$ in $I$ if and only if $x$ is optimal for $s^{\Gamma}$ in $I^{\Gamma} = (X,f^{\Gamma})$.
\end{proof}
Consequently, if every optimal solution set for $S$ is an approximation set in an instance $(X,f)$ of a $p$-objective optimization problem of type~$\Pi$, every optimal solution set for $S^{\Gamma}$ is an approximation set for the instance~$I^{\Gamma}$ of the $\Gamma$-transformed $p$-objective optimization problem with the same approximation quality:
\begin{corollary}\label{cor:transformed-supported-instance}
	Let~$S$ be a scalarization for~$\Pi$,
	let~$I = (X,f)$ be an instance of a $p$-objective optimization problem of type~$\Pi$, and let $\alpha \geq 1$. 
	Then, in~$I$, every optimal solution set for $S$ is an $\alpha$-approximation set if and only if, in the instance $I^{\Gamma}$ of the $\Gamma$-transformed optimization problem, every optimal solution set for $S^{\Gamma}$ is an $\alpha$-approximation set.
\end{corollary}
\begin{proof}
	Lemma~\ref{lem:transformed-scalarization-optimal} implies that an optimal solution set for $S$ in~$I$ is an optimal solution set for $S^{\Gamma}$, and vice versa. The set~$S^{\Gamma}$ is an $\alpha$-approximation set in~$I$ if and only if it is an $\alpha$-approximation set in~$I^{\Gamma}$.
\end{proof}\noindent
As a consequence of Corollary~\ref{cor:transformed-supported-instance}, we obtain the following transformation theorem:
\begin{theorem}[Transformation Theorem for Scalarizations with respect to Approximation]\label{thm:transformed-supported-general}
    Let $\alpha \geq 1$. Let $S$ be a scalarization for $\Pi = (\MIN, \MAX)$ such that, in each instance of each $p$-objective optimization problem of type~$\Pi$, every optimal solution set for~$S$ is an $\alpha$-approximation set. Then, for any other objective decomposition~$\Pi'$, there exists a scalarization~$S'$ such that the same holds true: in each instance of each $p$-objective optimization problem of type~$\Pi'$, every optimal solution set for~$S'$ is an $\alpha$-approximation set. 
\end{theorem}
\begin{proof}
    Let $\Pi' = (\MIN',\MAX')$ be an objective decomposition. Set $\Gamma = (\MIN \setminus \MIN') \cup (\MAX \setminus \MAX')$. Then, $(\Pi')^{\Gamma} = \Pi$. Let $I'$ be an instance of a multiobjective optimization problem of type $\Pi'$. Then, $(I')^{\Gamma}$ is an instance of the $\Gamma$-transformed optimization problem, which is of type $(\Pi')^{\Gamma} = \Pi$, and by assumption, in $(I')^{\Gamma}$, every optimal solution set for $S$ is an $\alpha$-approximation set. But then, in $\left((I')^{\Gamma}\right)^{\Gamma} = I'$, every optimal solution set for $S^{\Gamma}$ is an $\alpha$-approximation set by Corollary~\ref{cor:transformed-supported-instance}. Hence, $S' \coloneqq S^{\Gamma}$ is the desired scalarization for~$\Pi'$.
\end{proof}
\begin{example}\label{ex:weighted_sum_dual}
    Let $\Pi^{\min} = (\{1,\ldots,p\},\emptyset)$  and $\Pi^{\max} = (\emptyset,\{1,\ldots,p\})$.  Recall that the weighted sum scalarizing function for $\Pi^{\min}$ with weights $w_1,\ldots,w_p > 0$ is $s_w\colon \Rpg \to \R_>, s_w(y) = \sum_{i=1}^p w_i \cdot y_i$. Then, $-s_w$ is the weighted sum scalarizing function for $\Pi^{\max}$. Let $S = \{s_w \colon \Rpg \to \R\, \vert \, w \in \Rpg\}$ and $-S = \{- s_w \colon \Rpg \to \R\, \vert \, w \in \Rpg\}$ be the weighted sum scalarizations for $\Pi^{\min}$ and $\Pi^{\max}$, respectively. It is known that, in each instance of each $p$-objective minimization problem, every optimal solution set for $S$ is a $p$-approximation set, but there exist instances of $p$-objective maximization problems for which the set of $-S$-supported solutions does not yield any constant approximation quality~\citep{Bazgan+etal:power-weighted-sum,Glasser+etal:multi-hardness-proceedings}.
    
    Consider~$\Gamma = \{1,\ldots,p\}$. Then, $\Pi^{\max}$ is the $\Gamma$-transformed objective decomposition of $\Pi^{\min}$, and vice versa. Thus, the opposite result holds for the corresponding $\Gamma$-transformed scalarizations: The $\Gamma$-transformed scalarization of~$S$, which is a scalarization for $\Pi^{\max}$, is the scalarization
    \begin{align*}
        S^{\Gamma} = \left\{s^{\Gamma}_w\colon \Rpg \to \R, s^{\Gamma}_w(y) =  \frac{w_1}{y_1} + \ldots + \frac{w_p}{y_p} \; \bigg\vert \; w \in \Rpg \right\}.
    \end{align*} The $\Gamma$-transformed scalarization of $-S$, which is a scalarization for $\Pi^{\min}$, is the scalarization
    \begin{align*}
        -S^{\Gamma} = \left\{s^{\Gamma}_w\colon \Rpg \to \R, s^{\Gamma}_w(y) =  - \frac{w_1}{y_1} - \ldots - \frac{w_p}{y_p} \; \bigg\vert \; w \in \Rpg\right\}.
    \end{align*}
    Hence, in each instance of each $p$-objective maximization problem, every optimal solutions set for $S^{\Gamma}$ is a $p$-approximation set, but there exist instances of $p$-objective minimization problems for which the set of~$-S^{\Gamma}$-supported solutions does not yield any constant approximation quality.
\end{example}
\begin{remark}\label{rem:dual-alternatives}
	Let~$S$ be a scalarization for~$\Pi$ and let $\Gamma \subseteq \{1,\ldots,p\}$.
	In fact, for each~$s \in S$, any continuous strictly increasing function~$g\colon s(\R_>^p) \to \R$ could be utilized to define~$s^{\Gamma}$ via $s^{\Gamma}(y) \coloneqq g\left(s\left(\sigma^{\Gamma}(y)\right)\right)$ while still obtaining the results of Lemma~\ref{lem:transformed-scalarization-optimal}, Corollary~\ref{cor:transformed-supported-instance}, and Theorem~\ref{thm:transformed-supported-general}.
	However, defining~$s^{\Gamma}$ as in Definition~\ref{def:transformed-scalarization}
	yields the additional property that~$(s^{\Gamma})^{\Gamma} = s$ for any scalarizing function~$s$, i.e., applying the transformation twice yields the original scalarizing function. 
\end{remark}
\begin{example}\label{ex:weighted-sum-alternative-dual}
    Let $\Pi^{\min} = (\{1,\ldots,p\},\emptyset)$. Since the weighted sum scalarizing function~$s_w\colon \Rpg \rightarrow \R, s_w(y) = \sum_{i=1}^p w_i \cdot f_i(x)$ for $\Pi^{\min}$ is positive-valued, one can alternatively define its $\Gamma$-transformation with the help of $g\colon \R_> \rightarrow \R_>, g(t) = - \frac{1}{t}$. For a $p$-objective maximization problem, the corresponding optimization problem induced by this transformation reads as
    $$
        \min_{x \in X}\, - \frac{1}{\textcolor{black}{\sum_{i=1}^p} w_i \cdot \frac{1}{f_i(x)} }.
    $$
    Note that this single-objective optimization problem is equivalent to 
    $$
        \max_{x \in X}\; \sum_{j=1}^p w_j \cdot \frac{1}{\textcolor{black}{\sum_{i=1}^p} w_i \cdot \frac{1}{f_i(x)}}
    $$
    in the sense that, in each instance, the optimal solution sets coincide. The function~$h_w\colon \Rpg \to \R_>, h_w(y) = \sum_{j=1}^p w_j \cdot \frac{1}{\textcolor{black}{\sum_{i=1}^p} w_i \cdot \frac{1}{y_i}}$ is known as the \emph{weighted harmonic mean}~\citep{Ferger:weighted-harmonic-mean}.
\end{example}
\begin{example}\label{ex:max-ordering-dual}
    In each instance of each $p$-objective minimization problem, every optimal solution set for the \emph{weighted max-ordering scalarization}~$S = \{ s_w\colon \Rpg \rightarrow \R_>, s_w(y) = \max_{i=1,\ldots,p} w_i \cdot y_i \, \vert \, w \in \Rpg\}$ must contain at least one efficient solution for each nondominated image~\citep{Ehrgott:book}, i.e., every optimal solutions set for $S$ is a $1$-approximation set.
    The transformed scalarizing function~$s_w \in S$ for maximization (i.e., the $\{1,\ldots,p\}$-transformed scalarizing function) is
    $$ s^{\{1,\ldots,p\}}_w\colon \Rpg \rightarrow \R_>, s^{\{1,\ldots,p\}}_w(y) =  \max_{i=1,\ldots,p} w_i \cdot \frac{1}{y_i}.$$
    Consequently, the transformed scalarization of the weighted max-ordering scalarization for maximization is $S^{\{1,\ldots,p\}} = \{s_w^{\{1,\ldots,p\}}\colon \R^p \to \R \, \vert \, w \in \Rpg\}$ and, in each instance of each $p$-objective maximization problem, every optimal solution set for $S^{\{1,\ldots,p\}}$ is a $1$-approximation set. 
    
    For a $p$-objective maximization problem and a scalarizing function~$s^{\{1,\ldots,p\}}_w \in S^{\{1,\ldots,p\}}$, one can rewrite the corresponding implied single-objective optimization problem: In each instance, it holds that
    \begin{align*}
    \min_{x \in X} \max_{i=1,\ldots,p} w_i \cdot \frac{1}{f_i(x)} =  \min_{x \in X} \, \frac{1}{\min_{i=1,\ldots,p} \frac{1}{w_i} f_i(x)} = \frac{1}{\max_{x \in X}  \min_{i=1,\ldots,p} \frac{1}{w_i} f_i(x)}.
    \end{align*}
    Hence, the optimal solution set of $\min_{x \in X} \max_{i=1,\ldots,p} w_i \cdot \frac{1}{f_i(x)}$ coincides with the optimal solution set of $\max_{x \in X} \min_{i=1,\ldots,p} \tilde{w}_i \cdot f_i(x)$, where $\tilde{w} = (\frac{1}{w_1},\ldots, \frac{1}{w_p}) \in \Rpg$.
    This means that, in each instance of each $p$-objective maximization problem, instead of solving all single-objective minimization problem instances obtained from scalarizing functions~$s^{\{1,\ldots,p\}} \in S^{\{1,\ldots,p\}}$, one can solve the single-objective maximization problem instances obtained from the functions in $\{ r_{\tilde{w}}\colon \Rpg \rightarrow \R, r_{\tilde{w}}(y) = \min_{i=1,\ldots,p} \tilde{w}_i \cdot y_i \, \vert \, \tilde{w} \in \Rpg \}$ to obtain a $1$-approximation set.
\end{example}

\section{Conditions for General Scalarizations}\label{sec:conditions}
Given an objective decomposition~$\Pi$ and $\alpha \geq 1$, we study sufficient and necessary conditions for a scalarization~$S$ such that, in each instance of each multiobjective optimization problem of type~$\Pi$, the set of $S$-supported solutions is an $\alpha$-approximation set. We also derive upper bounds on the best approximation quality~$\alpha$ that can be achieved by $S$-supported solutions and that solely depends on the level sets of the scalarizing functions. In the following, we assume without loss of generality that the objective decomposition~$\Pi = (\MIN,\MAX)$ is given such that $\MIN = \{1,\ldots, k\}$ and $\MAX = \{k+1, \ldots, p\}$ holds for some $k \in \{0, \ldots,p\}$. Otherwise, the objectives may be reordered accordingly.

\medskip

The first result in this section states that, for any \emph{finite} set~$S$ of scalarizing functions and any $\alpha \geq 1$, the set of $S$-supported solutions, and, thus, any optimal solution set for $S$, is not an $\alpha$-approximation set in general.
\begin{theorem}\label{thm:finite-scalarizations}
	Let~$S$ be a scalarization of finite cardinality for~$\Pi$. Then, for any~$\alpha \geq 1$, there exists an instance~$I$ of a multiobjective optimization problem of type~$\Pi$ such that the set of $S$-supported solutions is not an $\alpha$-approximation set. 
\end{theorem}
\begin{proof}
    We first show that it suffices to construct an instance of a \emph{biobjective} optimization problem of each possible type such that the set of $S$-supported solutions is not an $\alpha$-approximation set. To this end, given an objective decomposition~$\Pi$ of $p>2$~objectives, we consider the objective decomposition~$\bar{\Pi}$ restricted to the objectives~$1$ and~$2$ given as $\bar{\Pi} \coloneqq (\emptyset,\{1,2\})$ if $k= 0$, $\bar{\Pi} \coloneqq (\{1\},\{2\})$ if $ k = 1$, and $\bar{\Pi} \coloneqq (\{1,2\},\emptyset)$ if $k\geq 2$. Now, let~$S$ be a scalarization of finite cardinality for~$\Pi$.
    Then, for each~$s \in S$, the function~$\bar{s} \colon \R^2_> \to \R$, $\bar{s}(y_1,y_2) \coloneqq s(y_1,y_2,1,\ldots,1)$ is a scalarizing function for~$\bar{\Pi}$. Applying the construction for $p = 2$ to the set $\bar{S} = \{\bar{s} \; \vert \; s\in S\}$ of scalarizing functions for $\bar{\Pi}$ then yields an instance of a biobjective optimization problem of type $\bar{\Pi}$ such that the set of $\bar{S}$-supported solutions is not an $\alpha$-approximation set; and this instance can be transformed into an instance of a $p$-objective optimization problem of type $\Pi$ such that the set of $S$-supported solutions is not an $\alpha$-approximation set by setting the additional $p-2$~objective functions to be equal to~$1$ for all~$x \in X$.
    
    \medskip
    
    It remains to show the claim for biobjective optimization problems, i.e., for $p=2$. To this end, we first consider the case $\bar{\Pi} = (\{1,2\},\emptyset)$, i.e., the case where both objective functions are to be minimized. Let~$S$ be a finite set of scalarizing functions for $\bar{\Pi}$. In the following, we construct an instance~$I = (X,f)$ of a biobjective minimization problem of type~$\bar{\Pi}$ whose feasible set~$X$ consists of $\vert S \rvert + 1$ solutions such that 
    \begin{enumerate}
    	\item\label{thm:finite-scalarizations-property1} no solution~$x \in X$ is $\alpha$-approximated by any other solution $x' \in X\setminus\{x\}$, and
    	\item\label{thm:finite-scalarizations-property2} we have $s(f(x)) \neq s(f(x'))$ for all~$x ,x' \in X$ with $x \neq x'$ and each~$s \in S$.
    \end{enumerate}
    We set $X = \{x^{(0)}, \ldots, x^{(\lvert S \rvert)}\}$ and inductively determine the components of the vectors~$f(x^{(\ell)})$ for $\ell = 0,\ldots, \vert S \rvert$ as follows: We start by setting $f_1(x^{(0)}) \coloneqq f_2(x^{(0)}) \coloneqq 1$. Next, let $f(x^{(0)}), \ldots, f(x^{(\ell-1)})$ be given for some $\ell \in \{1,\ldots, \vert S \rvert\}$. We construct the vector~$f(x^{(\ell)})$ such that~$x^{(\ell)}$ does not $\alpha$-approximate~$x^{(m)}$ and is not $\alpha$-approximated by~$x^{(m)}$ for~$m = 0,\ldots, \ell-1$, and such that~$s(f(x^{(\ell)})) \neq s(f(x^{(m)}))$ for $m = 0,\ldots, \ell-1$ and each $s \in S$.
    To this end, we first set
    \begin{align*}
        f_1(x^{(\ell)}) &\coloneqq \frac 1 {\alpha+1} \cdot \min\left\{f_1(x^{(0)}), \ldots, f_1(x^{(\ell-1)})\right\}\\
        f_2(x^{(\ell)}) &\coloneqq (\alpha + 1) \cdot \max\left\{f_2(x^{(0)}), \ldots, f_2(x^{(\ell-1)})\right\} + \vert S \rvert^2.
    \end{align*}
    If $s(f(x^{(\ell)})) \neq s(f(x^{(m)}))$ for $m = 0,\ldots, \ell-1$ and each $s \in S$, we are done. Otherwise, we do a decreasing step as follows: We strictly decrease the value~$f_1(x^{(\ell)})$ by a factor of~$\frac 1 2$ and strictly decrease the value of $f_2(x^{(\ell)})$ by an additive constant of~$1$. Note that, by strict monotonicity, this strictly decreases the value~$s(f(x^{(\ell)}))$ for each~$s \in S$. Thus, for each~$m \in \{0, \ldots, \ell-1\}$ and~$s \in S$ where we previously had $s(f(x^{(\ell)})) = s(f(x^{(m)}))$, we now have $s(f(x^{(\ell)})) < s(f(x^{(m)}))$. Note that this strict inequality is preserved in subsequent decreasing steps. Hence, after at most~$\ell \cdot \vert S \rvert$ many decreasing steps, we must have $s(f(x^{(\ell)})) \neq s(f(x^{(m)}))$ for $m = 0,\ldots, \ell-1$ and each $s \in S$, so we can proceed with the construction of~$f(x^{(\ell+1)})$ in iteration~$\ell + 1$.
    
    It is now left to prove that the resulting instance satisfies the two claimed Properties~\ref{thm:finite-scalarizations-property1} and \ref{thm:finite-scalarizations-property2}. The solution~$x^{(\ell)}$ whose objective values have been constructed in iteration~$\ell$ is not $\alpha$-approximated by~$x^{(m)}$ for~$m = 0,\ldots, \ell-1$ in the first objective~$f_1$ since
    \begin{align*}
        f_1(x^{(\ell)}) &\leq \frac 1 {\alpha+1} \cdot \min\left\{f_1(x^{(0)}), \ldots, f_1(x^{(\ell-1)}) \right\}\\ &< \frac 1 \alpha \cdot \min\left\{f_1(x^{(0)}), \ldots, f_1(x^{(\ell-1)})\right\}.
    \end{align*}
    Further, the solution~$x^{(\ell)}$ does not $\alpha$-approximate~$x^{(m)}$ in the second objective~$f_2$ for~$m = 0,\ldots, \ell-1$: We have performed at most $\ell \cdot \vert S \rvert \leq \vert S \rvert^2$ many decreasing steps, where, in each decreasing step, the value~$f_2(x^{(\ell)})$ has been decreased by~$1$. Thus,
    \begin{align*}
        f_2(x^{(\ell)}) &\geq (\alpha + 1) \cdot \max\left\{f_2(x^{(0)}), \ldots, f_2(x^{(\ell-1)})\right\} + \vert S \rvert^2 - \ell \cdot \vert S \rvert\\
        &>  \alpha \cdot \max\left\{f_2(x^{(0)}), \ldots, f_2(x^{(\ell-1)})\right\}.
    \end{align*}
    Hence, the instance~$I = (X,f)$ constructed as above indeed satisfies the two claimed Properties~\ref{thm:finite-scalarizations-property1} and \ref{thm:finite-scalarizations-property2}. Property~\ref{thm:finite-scalarizations-property2} implies that, for each scalarizing function~$s \in S$, exactly one solution is optimal for~$s$. Thus, at most~$\vert S \rvert$ many solutions can be $S$-supported, and at least one solution~$x \in X$ is not $S$-supported. However, by Property~\ref{thm:finite-scalarizations-property1}, this solution~$x$ is not $\alpha$-approximated by any other solution. Thus, $I$ is an instance of a biobjective minimization problem for which the set of $S$-supported solutions is not an $\alpha$-approximation set.
    
    \medskip
    
    In order to show the claim for the case $\bar{\Pi} = (\emptyset, \{1,2\})$, i.e., the case where both objective functions are to be maximized, we apply the above construction to the ${\{1,2\}}$-transformed scalarization~$S^{\{1,2\}}$. This yields an instance~$I$ of a biobjective minimization problem where the set of $S^{{\{1,2\}}}$-supported solutions is not an $\alpha$-approximation set. Thus, by Corollary~\ref{cor:transformed-supported-instance}, the set of $S$-supported solutions is not an $\alpha$-approximation set in the ${\{1,2\}}$-transformed instance~$I^{{\{1,2\}}}$, which is an instance of a biobjective maximization problem. The case $\bar{\Pi} = (\{1\}, \{2\})$ follows analogously with the transformation induced by~$\Gamma = \{2\}$.
\end{proof}
Note that the (algorithmically motivated) approximation for instances of $p$-objective minimization problems in~\cite{Glasser+etal:multi-hardness-proceedings,Glasser+etal:multi-hardness,Bazgan+etal:power-weighted-sum} is done by an instance-based choice of finitely many scalarizing functions. Nevertheless, to obtain a scalarization that yields approximation sets for arbitrary instances of arbitrary $p$-objective minimization problems, the cardinality of~$S$ must be infinite by Theorem~\ref{thm:finite-scalarizations}.
However, the inapproximability results for maximization problems presented in~\cite{Bazgan+etal:power-weighted-sum,Helfrich+etal:cones} state that there exists an instance where even the set of all supported solutions (for the weighted sum scalarization) does not constitute an approximation set. Hence,
in general, even considering infinitely many scalarizing functions is not sufficient for approximation. Instead, additional conditions for the scalarizing functions are crucial, which we derive next.

We first study with what approximation quality a given feasible solution can be approximated by optimal solutions for a single scalarizing function. Afterwards, we investigate what approximation quality can be achieved by every optimal solution set for a scalarization~$S$, and then derive conditions under which an optimal solution set for $S$ constitutes an approximation for arbitrary instances of $p$-objective optimization problems of type~$\Pi$.

\medskip

The first result shows that, given a feasible solution~$x'$, if the component-wise maximum ratio between points in the level set of a scalarizing function at~$f(x')$ can be bounded by some $\alpha \geq 1$, then~$x'$ is $\alpha$-approximated by every optimal solution for the scalarizing function:

\begin{lemma}\label{lem:two-solutions-sufficient}
	In an instance of a $p$-objective optimization problem of type~$\Pi$, let~$x' \in X$ and $y' \coloneqq f(x')$. Let $s\colon \Rpg \to \R$ be a scalarizing function for $\Pi$ such that the level set~$L(y',s)$
	is bounded from above in $i = 1, \ldots, k$  by some $q \in \Rpg$ and bounded from below in $i = k+1, \ldots,p$ by some $q' \in \Rpg$. Then the solution~$x'$ is $\alpha$-approximated by every solution~$x \in X$ that is optimal for $s$, where
	\begin{align}
	\alpha \coloneqq \sup\left\{\max \left\{ \frac{y_1}{y'_1}, \ldots, \frac{y_k}{y'_k}, \frac{y'_{k+1}}{y_{k+1}},\ldots,\frac{y'_{p}}{y_{p}} \right\} \; \bigg\vert \; y \in L(y',s)
	\right\}.\label{eq:alphadef-onepoint}
	\end{align}
\end{lemma}
\begin{proof}
	Note that $\alpha < \infty$ since $y'=f(x')$ is fixed and $L(y',s)$ is bounded from \textcolor{black}{above} in $i =1,\ldots,k$ by $q \in \Rpg$ and bounded from \textcolor{black}{below} in $i = k+1, \ldots,p$ by $q' \in \Rpg$. Let~$x$ be an optimal solution for~$s$. 
	By Lemma~\ref{lem:beam}, there exists~$\lambda\in\R_>$ such that $y'' \in \Rpg$ defined by $y''_i \coloneqq \lambda \cdot f_i(x)$ for $i = 1, \ldots, k$ and $y''_i \coloneqq \frac{1}{\lambda} \cdot f_i(x)$ for $i =k+1,\ldots,p$ satisfies $s(y'') = s(y')$, i.e., $y'' \in L(y',s)$.  Moreover, $\lambda < 1$ would imply that $s(y'') < s(f(x)) \leq s(f(x')) = s(y') = s(y'')$ by strict monotonicity of~$s$ and optimality of~$x$ for~$s$, which is a contradiction. Hence, $\lambda\geq1$, so $f(x) \leq_{\Pi} y''$, and we obtain
	\begin{align*}
    	\frac{f_i(x)}{f_i(x')} \leq \frac{y''_i}{f_i(x')} = \frac{y''_i}{y'_i} \leq \max \left\{ \frac{y''_1}{y'_1}, \ldots, \frac{y''_k}{y'_k}, \frac{y'_{k+1}}{y''_{k+1}},\ldots,\frac{y'_{p}}{y''_{p}} \right\} \leq \alpha
	\end{align*}
	for~$i =1,\ldots,k$ and
	\begin{align*}
    	\frac{f_i(x')}{f_i(x)} \leq \frac{f_i(x')}{y''_i} = \frac{y'_i}{y''_i} \leq \max \left\{ \frac{y''_1}{y'_1}, \ldots, \frac{y''_k}{y'_k}, \frac{y'_{k+1}}{y''_{k+1}},\ldots,\frac{y'_{p}}{y''_{p}} \right\} \leq \alpha
	\end{align*}
	for $i = k+1, \ldots, p$.
\end{proof}

\medskip

We proceed by investigating with what approximation quality a given feasible solution can be approximated by the set of $S$-supported solutions of a scalarization~$S$.
\begin{proposition}\label{prop:scalarizationset-necessary}
	In an instance of a $p$-objective optimization problem of type~$\Pi$, let~$x' \in X$ be given. Let $S$ be a scalarization for $\Pi$ such that, for some scalarizing function~$\bar{s} \in S$, the level set~$L(y',\bar{s})$
	for $y' \coloneqq f(x')$ is bounded from above in $i=1,\ldots,k$ by some $q \in \Rpg$ and bounded from below in $i=k+1,\ldots,p$ by some $q' \in \Rpg$.
	Then, for any $\varepsilon > 0$, the solution~$x'$ is $(\alpha+ \varepsilon)$-approximated by every optimal solution for some scalarizing function~$s \in S$, where
	\begin{align*}
    	\alpha \coloneqq \inf_{s \in S}\quad \sup\left\{\max \left\{ \frac{y_1}{y'_1}, \ldots, \frac{y_k}{y'_k}, \frac{y'_{k+1}}{y_{k+1}},\ldots,\frac{y'_{p}}{y_{p}} \right\} \; \bigg\vert \; y \in L(y',s)\right\}.
	\end{align*}
	If the infimum is attained at some~$s\in S$, then~$x'$ is $\alpha$-approximated by every optimal solution for~$s$.
\end{proposition}
\begin{proof} 
    Note that $\alpha < \infty$ since $y'=f(x')$ is fixed and $L(y',\bar{s})$ is bounded from above in $i=1,\ldots,k$ by some $\bar{q} \in \Rpg$ and bounded from below in $i=k+1,\ldots,p$ by some $\bar{q}' \in \Rpg$.
	Given $\varepsilon> 0$, let $s \in S$ be a scalarizing function such that
	\begin{align*}
    	\sup\left\{\max \left\{ \frac{y_1}{y'_1}, \ldots, \frac{y_k}{y'_k}, \frac{y'_{k+1}}{y_{k+1}},\ldots,\frac{y'_{p}}{y_{p}} \right\} \; \bigg\vert \; y \in L(y',s)\right\} \leq \alpha + \varepsilon.
	\end{align*}
	Then $L(y',s)$ must be bounded from above in $i=1,\ldots,k$ by some $q \in \Rpg$ and bounded from below in $i=k+1,\ldots,p$ by some $q' \in \Rpg$. Thus, Lemma~\ref{lem:two-solutions-sufficient} implies that $x'$ is $(\alpha + \varepsilon)$-approximated by any solution~$x \in X$ that is optimal for~$s$, which proves the first claim. If the infimum is attained at~$s\in S$, this means that we even have
	\begin{align*}
    	\sup\left\{\max \left\{ \frac{y_1}{y'_1}, \ldots, \frac{y_k}{y'_k}, \frac{y'_{k+1}}{y_{k+1}},\ldots,\frac{y'_{p}}{y_{p}} \right\} \; \bigg\vert \; y \in L(y',s)\right\} = \alpha,
	\end{align*} 
	and the second claim also follows immediately by using Lemma~\ref{lem:two-solutions-sufficient}.
\end{proof}\noindent
If the scalarization~$S$ admits a common finite upper bound on $$\displaystyle\inf_{s \in S}\quad \sup\left\{\max \left\{ \frac{y_1}{y'_1}, \ldots, \frac{y_k}{y'_k}, \frac{y'_{k+1}}{y_{k+1}},\ldots,\frac{y'_{p}}{y_{p}} \right\} \; \bigg\vert \; y \in L(y',s)\right\}$$ for all points~$y' \in \Rpg$, then Proposition~\ref{prop:scalarizationset-necessary} implies that, in each instance of each $p$-objective optimization problem of type~$\Pi$, every optimal solution set for~$S$ yields a constant approximation quality:
\begin{theorem}\label{cor:sup-inf-sup}
	Let~$S$ be a scalarization for $\Pi$ and let
	\begin{align*}
    	\alpha \coloneqq \sup_{y' \in \Rpg} \quad \inf_{s \in S}\quad \sup\left\{\max \left\{ \frac{y_1}{y'_1}, \ldots, \frac{y_k}{y'_k}, \frac{y'_{k+1}}{y_{k+1}},\ldots,\frac{y'_{p}}{y_{p}} \right\} \; \bigg\vert \; y \in L(y',s)\right\}.
	\end{align*}
	If $\alpha < \infty$, then, in each instance of each $p$-objective optimization problem of type~$\Pi$, every optimal solution set for $S$ is an~$(\alpha + \varepsilon)$-approximation set for any $\varepsilon > 0$. If, additionally, the infimum
	\begin{align*}
    	\inf_{s \in S}\quad \sup\left\{\max \left\{ \frac{y_1}{y'_1}, \ldots, \frac{y_k}{y'_k}, \frac{y'_{k+1}}{y_{k+1}},\ldots,\frac{y'_{p}}{y_{p}} \right\} \; \bigg\vert \; y \in L(y',s)\right\} 
	\end{align*}
	is attained and finite for each $y' \in \Rpg$, then, in each instance of each $p$-objective optimization problem of type~$\Pi$, every optimal solution set for $S$ is an~$\alpha$-approximation set.
\end{theorem}
\begin{proof}
    Let~$I = (X,f)$ be an instance of a $p$-objective optimization problem of type~$\Pi$. Let~$x' \in X$ be a feasible solution and set~$y'\coloneqq f(x')$. Then
    \begin{align*}
        \inf_{s \in S}\quad \sup\left\{\max \left\{ \frac{y_1}{y'_1}, \ldots, \frac{y_k}{y'_k}, \frac{y'_{k+1}}{y_{k+1}},\ldots,\frac{y'_{p}}{y_{p}} \right\} \; \bigg\vert \; y \in L(y',s)\right\} \leq  \alpha,
    \end{align*}
    which implies that $L(y',\bar{s})$ must be bounded from above in $i=1,\ldots, k$ by some $\bar{q} \in \Rpg$ and bounded from below in $i=k+1,\ldots p$ by some $\bar{q}' \in \Rpg$ for at least one~$\bar{s}\in S$. Consequently, the first claim follows by Proposition~\ref{prop:scalarizationset-necessary}. The second claim follows similarly by using the second statement in Proposition~\ref{prop:scalarizationset-necessary}.
\end{proof}\noindent
Given a scalarization~$S$ for $\Pi$ and $y' \in \Rpg$, set 
\begin{align*}
    \alpha(y') \coloneqq \inf_{s \in S}\quad \sup\left\{\max \left\{ \frac{y_1}{y'_1}, \ldots, \frac{y_k}{y'_k}, \frac{y'_{k+1}}{y_{k+1}},\ldots,\frac{y'_{p}}{y_{p}} \right\} \; \bigg\vert \; y \in L(y',s)\right\}.
\end{align*}
Theorem~\ref{cor:sup-inf-sup} states that, if there exists a common finite upper bound~$\alpha$ with $\sup_{y' \in \Rpg} \alpha(y) \leq \alpha < \infty$, then a constant approximation quality (namely $\alpha+\varepsilon$) is achieved by every optimal solution set for $S$ in arbitrary instances of arbitrary $p$-objective optimization problems of type~$\Pi$. The following example, however, shows that the weaker condition $\alpha(y')<\infty$ for every~$y'\in\Rpg$ (which holds if all level sets~$L(y',s)$ are bounded from above in $i=1,\ldots, k$ by some $\bar{q} \in \Rpg$ and bounded from below in $i=k+1,\ldots p$ by some $\bar{q}' \in \Rpg$) is \emph{not} sufficient in order to guarantee a constant approximation quality:
\begin{example}
     Let $p = 2$ and $\Pi^{\min} = (\{1,2\},\emptyset)$. Consider the scalarization~$S = \{s\}$ for $\Pi^{\min}$, where $s\colon\R^2_>\rightarrow \R$ is defined by $s(y) \coloneqq \min \{y_1^2 + y_2, y_1+ y_2^2\}$. Then, Theorem~\ref{thm:finite-scalarizations} shows that there exists an instance of a biobjective minimization problem with a solution~$x' \in X$ such that $x'$ is not $\alpha$-approximated by any $S$-supported solution. Nevertheless, for any $y' \in \R^2_>$, it can be shown that
    \begin{align*}
        \alpha(y') &= \sup\left\{\max\left\{\frac{y_1}{y'_1},\frac{y_2}{y'_2}\right\} \; \bigg\vert \; y \in L(y',s)\right\} \\
        &\leq \max \left\{\frac{s(y')}{y'_1},\frac{\sqrt{s(y')}}{y'_1},\frac{s(y')}{y'_2},\frac{\sqrt{s(y')}}{y'_2}\right\} < \infty.
    \end{align*}
    Further,
    \begin{align*}
        \sup_{y' \in \Rpg} \alpha(y') &\geq \sup_{a \geq 1} \quad \sup\left\{ \max\left\{\frac{y_1}{a},\frac{y_2}{a}\right\} \; \bigg\vert \; y \in \R^2_>, \; s(y) = s((a,a))\right\} \\
        &= \sup_{a \geq 1} \quad \sup\left\{ \max\left\{\frac{y_1}{a},\frac{y_2}{a} \right\} \; \bigg\vert \; y \in \R^2_>, \; s(y) = a^2 + a\right\}\\
        &= \sup_{a \geq 1}\frac{a^2 + a}{a} =  \sup_{a \geq 1} a+ 1 = \infty,
    \end{align*}
    so there is no $\alpha < \infty$ with $\sup_{y' \in \Rpg} \alpha(y') \leq \alpha$ as required in Theorem~\ref{cor:sup-inf-sup}.
\end{example}
There exist scalarizations~$S$ for which the approximation quality $\alpha$ given in Theorem~\ref{cor:sup-inf-sup} is tight in the sense that, for any $\varepsilon > 0$, there is an instance of a multiobjective optimization problem of type~$\Pi$ such that the set of $S$-supported solutions is not an $(\alpha \cdot (1 - \varepsilon))$-approximation set. Examples of such scalarizations where, additionally, $\alpha$ is easy to calculate, are presented in Section~\ref{sec:tightness-norm-based-scalarizations}.

Nevertheless, we now show that the approximation quality in Theorem~\ref{cor:sup-inf-sup} is not tight in general. To this end, we provide an example of a scalarization~$S$ for minimization for which each individual scalarizing function~$s \in S$ does not satisfy the requirements of Lemma~\ref{lem:two-solutions-sufficient}. That is, for each point~$y' \in \Rpg$, the level set~$L(y',s)$ is not bounded from above in some~$i=1,\ldots,p$ .
However, for each instance, every optimal solution set for the whole scalarization~$S$ is indeed a $1$-approximation set.
\begin{example}\label{ex:approx-not-tight-in-general}
    Again, let $p=2$ and $\Pi^{\min} = (\{1,2\},\emptyset)$.
	For each $w \in \R^2_>$ and $\varepsilon \in (0,1)$, define a scalarizing function~$s_{w,\varepsilon}\colon \R_>^2 \to \R$ for $\Pi^{\min}$ via
	\begin{align*}
	s_{w,\varepsilon}(y) \coloneqq \max\Bigg\{ \min \left\{ \frac{w_1\cdot y_1}{\varepsilon}, w_2 \cdot y_2 \right\} ,\min \left\{w_1 \cdot y_1, \frac{w_2\cdot y_2}{\varepsilon} \right\}\Bigg\}.
	\end{align*}
	Then, the level set $L(y',s_{w,\varepsilon})=\left\{y \in \R^2_> \, \vert \, s_{w,\varepsilon}(y) = s_{w,\varepsilon}(y')\right\}$ is unbounded for each $y' \in \R^2_>$, and consequently not bounded from above in neither $i=1$ nor $i=2$. 
	Thus, for each $s_{w,\varepsilon}$ and each $y' \in \R_>$, it holds that
	\begin{align*}
	    \sup \left\{ \max \left\{ \frac{y_1}{y'_1}, \frac{y_2}{y'_2} \right\} \; \bigg\vert \; y \in L(y',s_{w,\varepsilon})\right\} = \infty
	\end{align*}
	and, therefore, the value~$\alpha$ given in Theorem~\ref{cor:sup-inf-sup} is infinite.
	However, for $S = \{s_{w,\varepsilon} : w \in \R^2_> \; \vert \; 0<\varepsilon<1\}$, in any instance~$I = (X,f)$ of any biobjective minimization problem, at least one corresponding efficient solution~$x \in X_E$ for every nondominated image $y \in Y_N$ must be contained in every optimal solution set for~$S$ and, consequently, every optimal solution set is a $1$-approximation set:
	Since~$f(X)$ is a compact subset of~$\R^2_>$, it is bounded from above in each~$i$ by some~$y \in \R^2_>$, and from below in each~$i$ by some~$y' \in \R^2_>$. Choose~$\varepsilon < \frac{y'_1 \cdot y'_2}{y_1 \cdot y_2}$. Then, for each $w \in \R^2_>$ with~$y'_2 \leq w_1 \leq y_2$ and $y'_1 \leq w_2 \leq y_1$, and each $x\in X$, we have
	\begin{align*}
    	&w_1 \cdot f_1(x) \leq y_2 \cdot y_1 < \frac{y'_1 \cdot y'_2}{\varepsilon} \leq \frac{w_2 \cdot f_2(x)}{\varepsilon} \quad\text{ and }\\
    	&w_2 \cdot f_2(x) \leq y_1 \cdot y_2 < \frac{y'_2 \cdot y'_1}{\varepsilon} \leq \frac{w_1 \cdot f_1(x)}{\varepsilon},
	\end{align*}
	so $s_{w,\varepsilon}(f(x)) = \max\{w_1 \cdot f_1(x),w_2\cdot f_2(x)\}$. This means that, for such combinations of~$w$ and~$\varepsilon$, the scalarizing function~$s_{w,\varepsilon}$ coincides with the weighted max-ordering scalarizing function.
	It is well-known that, for $y \in Y_N$, any optimal solution~$x$ for the weighted max-ordering scalarizing function with weights~$w_1= y_2$ and $w_2 = y_1$ is a preimage of~$y$, i.e., $f(x) = y$ and $x \in X_E$.
\end{example}

\section{Weighted Scalarizations}\label{sec:specific-scalarizing}

In this section, we tailor the results of Section~\ref{sec:conditions} to so-called \emph{weighted scalarizations}, in which the objective functions are weighted by positive scalars before a given scalarizing function~$s$ for an objective decomposition~$\Pi$ is applied. By varying the weights, different optimal solutions are potentially obtained. In Section~\ref{sec:simplified-approximation-quality}, we show that the computation of the approximation quality $\alpha$ given in Theorem~\ref{cor:sup-inf-sup} simplifies for weighted scalarizations. Moreover, we see in Section~\ref{sec:tightness-norm-based-scalarizations} that $\alpha$ is easy to calculate and is best possible for all norm-based weighted scalarizations applied in the context of multiobjective optimization.

\medskip

As in Section~\ref{sec:conditions}, we assume without loss of generality that the objective decomposition is given as~$\Pi = (\{1,\ldots,k\},\{k+1,\ldots,p\})$ for some $k \in \{0,\ldots, p\}$. Weighted scalarizations for $\Pi$ are then formally defined as follows:
\begin{definition}
    Let~$W \subseteq \Rpg$ be a \emph{set of possible weights} and $s: \Rpg \to \R$ some scalarizing function for $\Pi$. Then, the \emph{weighted scalarization~$S$} induced by $W$ and $s$ is defined via
    \begin{align}\label{eq:weighted-definition}
        S = \left\{s_w \colon \R_>^p \to \R, s_w(y) = s(w_1\cdot y_1, \ldots, w_p\cdot y_p) \;\vert\; w \in W \right\}.
    \end{align}
\end{definition}
As the most prominent example, this class contains the weighted sum scalarization, where $W = \Rpg$ and $s\colon\Rpg \to \R, s(y) \coloneqq \sum_{i =1}^k y_i - \sum_{i=k+1}^p y_i$, see~Example~\ref{ex:weighted_sum}. 

\subsection{Simplified Computation of the Approximation Quality}\label{sec:simplified-approximation-quality}
 
For weighted scalarizations~$S$ as in~\eqref{eq:weighted-definition}, the computation of the approximation quality~$\alpha$ given in Theorem~\ref{cor:sup-inf-sup} simplifies as follows:
\begin{lemma}\label{lem:sup-inf-sup-for-weighted}
	Let~$S$ be the weighted scalarization induced by $W = \Rpg$ and some scalarizing function~$s$ for $\Pi$. Define $\alpha \geq 1$ as in \eqref{eq:alphadef-onepoint}, i.e.,
	\begin{align*}
    	\alpha \coloneqq \sup_{y' \in \Rpg} \quad \inf_{s' \in S}\quad \sup\left\{\max \left\{ \frac{y_1}{y'_1}, \ldots, \frac{y_k}{y'_k}, \frac{y'_{k+1}}{y_{k+1}},\ldots,\frac{y'_{p}}{y_{p}} \right\} \; \Bigg\vert \; y \in L(y',s')\right\}.
	\end{align*}
	Further, define $\beta \geq 1$ by
	\begin{align*}
    	\beta \coloneqq \inf_{\bar{y} \in \Rpg} \ \sup \left\{\max \left\{ \frac{y^*_1}{\bar{y}_1}, \ldots, \frac{y^*_k}{\bar{y}_k}, \frac{\bar{y}_{k+1}}{y^*_{k+1}},\ldots,\frac{\bar{y}_{p}}{y^*_{p}} \right\}\; \Bigg\vert \; y^* \in L(\bar{y},s)\right\}.
	\end{align*}
	Then, it holds that~$\alpha = \beta$.
\end{lemma}
\begin{proof}
	Let~$y' \in \Rpg$. For each~$s' \in S$, there exists a vector~$w\in \Rpg$ of parameters such that $s' = s_w$. Vice versa, for each vector~$w\in \Rpg$ of parameters, there exists a scalarizing function~$s' \in S$ such that $s_w = s'$. Consequently,
	\begin{align*}
		&\inf_{s' \in S}\quad \sup\left\{\max \left\{ \frac{y_1}{y'_1}, \ldots, \frac{y_k}{y'_k}, \frac{y'_{k+1}}{y_{k+1}},\ldots,\frac{y'_{p}}{y_{p}} \right\} \; \Bigg\vert \; y \in L(y',s')\right\}\\
		&= \inf_{w \in \Rpg}\quad \sup\left\{\max \left\{ \frac{y_1}{y'_1}, \ldots, \frac{y_k}{y'_k}, \frac{y'_{k+1}}{y_{k+1}},\ldots,\frac{y'_{p}}{y_{p}} \right\} \; \Bigg\vert \; y \in L(y',s_w)\right\}.
	\end{align*}
	Further, it holds that
	\begin{align*}
		\inf_{w \in \Rpg}\quad \sup \Bigg\{ \max \bigg\{ \frac{y_1}{y'_1}, \ldots, \frac{y_k}{y'_k}, &\frac{y'_{k+1}}{y_{k+1}},\ldots,\frac{y'_{p}}{y_{p}} \bigg\} \; \bigg\vert \; y \in L(y',s_w) \Bigg\} \\
		=\inf_{w \in \Rpg} \quad \sup \Bigg\{ \max \bigg\{ \frac{y_1}{y'_1}, \ldots, \frac{y_k}{y'_k}, &\frac{y'_{k+1}}{y_{k+1}},\ldots,\frac{y'_{p}}{y_{p}} \bigg\} \; \bigg\vert \; \\
		y \in \Rpg, s(w_1\,\cdot\,&y_1, \ldots, w_p \cdot y_p) =  s(w_1 \cdot y'_1, \ldots, w_p \cdot y'_p) \Bigg\}  \\
		=\inf_{\bar{y} \in \Rpg}\quad \sup \Bigg\{ \max \bigg\{ \frac{y_1}{y'_1}, \ldots, \frac{y_k}{y'_k}, &\frac{y'_{k+1}}{y_{k+1}},\ldots,\frac{y'_{p}}{y_{p}} \bigg\} \; \bigg\vert \; \\
	    &y \in \Rpg, s\left(\frac{\bar{y}_1}{y'_1} \cdot y_1, \ldots, \frac{\bar{y}_p}{y'_p} \cdot y_p\right) =  s(\bar{y}) \Bigg\}  \\
		=\inf_{\bar{y} \in \Rpg}\quad \sup\Bigg\{\max \bigg\{ \frac{y^*_1}{\bar{y}_1}, \ldots, \frac{y^*_k}{\bar{y}_k}, &\frac{\bar{y}_{k+1}}{y^*_{k+1}},\ldots,\frac{\bar{y}_{p}}{y^*_{p}}\bigg\} \; \bigg\vert \; y^* \in \Rpg, s(y^*) =  s(\bar{y})\Bigg\}  
		= \beta,
	\end{align*}
	where we substitute~$\bar{y}_i = w_i \cdot y'_i$, $i = 1, \ldots,p$, in the second equality and 
	$y^*_i = \frac{\bar{y}_i}{y'_i} \cdot y_i$, $i=1,\ldots,p$, in the third equality. \textcolor{black}{Note that, since $y'_i>0$ for $i=1,\ldots,p$, every point~$\bar{y}\in\Rpg$ can actually be obtained via~$\bar{y}_i = w_i \cdot y'_i$ using an appropriate positive weight vector~$w\in\Rpg$.} Hence, for each~$y' \in \Rpg$, the value
	\begin{align*}
		\inf_{s' \in S}\quad \sup\left\{\max \left\{ \frac{y_1}{y'_1}, \ldots, \frac{y_k}{y'_k}, \frac{y'_{k+1}}{y_{k+1}},\ldots,\frac{y'_{p}}{y_{p}} \right\} \; \bigg\vert \; y \in L(y',s')\right\}
	\end{align*}
	is equal to the constant~$\beta$, and we obtain $\alpha = \sup_{y' \in \Rpg} \beta = \beta$.
\end{proof}
Consequently, if, for some $\bar{y} \in \Rpg$, the level set~$L(\bar{y},s)$ is bounded from above in $i = 1, \ldots,k$ by some $q \in \Rpg$ and bounded from below in $i=k+1,\ldots,p$ by some $q' \in \Rpg$, Theorem~\ref{cor:sup-inf-sup} and Lemma~\ref{lem:sup-inf-sup-for-weighted} imply that, in each instance, every optimal solution set for~$S$ constitutes an approximation set with approximation quality arbitrarily close or even equal to~$\beta$, with~$\beta$ computed as in Lemma~\ref{lem:sup-inf-sup-for-weighted}. This is captured in the following theorem:
\begin{theorem}\label{thm:condition-weighted-scalarization}
    Let~$S$ be the weighted scalarization induced by $W = \Rpg$ and some scalarizing function~$s$ for $\Pi$ such that, additionally, $L(\bar{y},s)$ is bounded from above in $i = 1, \ldots,k$ by some $q \in \Rpg$ and bounded from below in $i=k+1,\ldots,p$ by some $q' \in \Rpg$ for some $\bar{y} \in \Rpg$. Define
	\begin{align*}
	    \beta \coloneqq \inf_{\bar{y} \in \Rpg} \ \sup \left\{\max \left\{ \frac{y^*_1}{\bar{y}_1}, \ldots, \frac{y^*_k}{\bar{y}_k}, \frac{\bar{y}_{k+1}}{y^*_{k+1}},\ldots,\frac{\bar{y}_{p}}{y^*_{p}} \right\}\; \bigg\vert \; y^* \in L(\bar{y},s)\right\}.
    \end{align*}
    Then, in each instance of each $p$-objective optimization problem of type~$\Pi$, every optimal solution set for~$S$ is a $(\beta + \varepsilon)$-approximation set for any $\varepsilon > 0$. If the infimum is attained, i.e., if there exists $\bar{y} \in \Rpg$ such that
	\begin{align*}
	    \beta = \sup \left\{\max \left\{ \frac{y^*_1}{\bar{y}_1}, \ldots, \frac{y^*_k}{\bar{y}_k}, \frac{\bar{y}_{k+1}}{y^*_{k+1}},\ldots,\frac{\bar{y}_{p}}{y^*_{p}} \right\}\; \bigg\vert \; y^* \in L(\bar{y},s)\right\}
	\end{align*}
	holds, then, in each instance of each $p$-objective optimization problem of type~$\Pi$, every optimal solution set for $S$ is a $\beta$-approximation set.\qed
\end{theorem}

\begin{example}
    Again, consider the objective decomposition~$\Pi^{\min} = (\{1,\ldots, p\}, \emptyset)$ and the scalarizing function~$s\colon\Rpg \rightarrow \R, s(y) = \sum_{i=1}^p y_i$. Then, the weighted scalarization induced by $W = \Rpg$ and $s$ is the weighted sum scalarization~$S = \{ s_w\colon \Rpg \rightarrow \R, s_w(y) = \sum_{i=1}^p w_i y_i  \; \vert \; w \in \Rpg\}$ for minimization. For each~$\bar{y}\in \Rpg$, it can be shown (see Lemma~\ref{lem:sup-norm} in the appendix) that a tight upper bound on the component-wise worst case ratio of $\bar{y}$ to any $y^* \in L(\bar{y},s)$ is
    \begin{align*}
        \sup \left\{\max\left\{ \frac{y^*_1}{\bar{y}_1},\ldots, \frac{y^*_p}{\bar{y}_p} \right\} \; \bigg\vert \; y^*\in\Rpg, s(y^*) = s(\bar{y}) \right\} = \left(\sum_{j=1}^p \bar{y}_j \right) \cdot \max_{i=1,\ldots,p} \left\{ \frac{1}{\bar{y}_i} \right\}.
    \end{align*}
    For~$p=2$, this is illustrated in Figure~\ref{fig:level-sets-weighted-sum} (left).
    Since $(1,\ldots,1) \in \Rpg$ with
	\begin{align*}
	   &\inf_{\bar{y} \in \Rpg} \quad \sup \left\{\max_{i=1,\ldots,p}\left\{ \frac{y^*_i}{\bar{y}_i} \right\} \; \bigg\vert \; y^*\in\Rpg, s(y^*) = s(\bar{y}) \right\} \\
	  \geq\ &\sup \left\{\max_{i=1,\ldots,p}\left\{ \frac{y_i}{1} \right\}\; \bigg\vert \; y\in\Rpg, s(y) = s((1,\ldots,1)) \right\}\\
	   = \ &\left(\sum_{j=1}^p 1 \right) \max_{i=1,\ldots,p} \left\{ \frac{1}{1} \right\} = p,
	\end{align*}
	where the proof of the first inequality is given in Theorem~\ref{thm:norms-minimization}, the approximation quality for the weighted sum scalarization for minimization given in Theorem~\ref{thm:condition-weighted-scalarization} resolves to~$\beta = p$.
	
	In view of Theorem~\ref{cor:sup-inf-sup}, observe that, for each~$y' \in \Rpg$, exactly the parameter vector $w' = \left( \frac{\sum_{i=1}^p y'_i}{y'_1}, \ldots, \frac{\sum_{i=1}^p y'_i}{y'_p} \right) \in \Rpg$ satisfies 
	\begin{align*}
	    &\inf_{s_w \in S} \sup\left\{\max \left\{ \frac{y_1}{y'_1}, \ldots, \frac{y_p}{y'_p} \right\} \; \bigg\vert \;y \in L(y',s_{w})\right\} \\
	    = &\ \sup\left\{\max \left\{ \frac{y_1}{y'_1}, \ldots, \frac{y_p}{y'_p} \right\} \; \bigg\vert \;y \in L(y',s_{w'}) \right\} = p,
	\end{align*}
	see Figure~\ref{fig:level-sets-weighted-sum} (right) for an illustration of the case $p=2$. Hence, Theorems~\ref{cor:sup-inf-sup} and~\ref{thm:condition-weighted-scalarization} indeed generalize the known approximation results on the weighted sum scalarization for minimization in~\cite{Glasser+etal:multi-hardness-proceedings,Bazgan+etal:power-weighted-sum}. In fact, the known tightness of these results yields that the approximation quality in Theorems~\ref{cor:sup-inf-sup} and~\ref{thm:condition-weighted-scalarization} is tight for the weighted sum scalarization for minimization.
	
	For the weighted sum scalarization for objective decompositions~$\Pi$ containing maximization objectives, however, it can be shown that 
	\begin{align*}
	    \inf_{\bar{y} \in \Rpg} \ \sup \left\{\max \left\{ \frac{y^*_1}{\bar{y}_1}, \ldots, \frac{y^*_k}{\bar{y}_k}, \frac{\bar{y}_{k+1}}{y^*_{k+1}},\ldots,\frac{\bar{y}_{p}}{y^*_{p}} \right\}\; \Bigg\vert \; y^* \in L(\bar{y},s)\right\} = \infty.
	\end{align*}
	This bound is also tight: for every~$\alpha \geq 1$, an instance of a $p$-objective optimization problem of type~$\Pi$ exists for which the set of supported solutions is not an $\alpha$-approximation set. A proof is given in Section~\ref{sec:tightness-norm-based-scalarizations}. 
    \end{example}
    \begin{figure}[tb]
        \centering
        \includegraphics[width = 1 \textwidth]{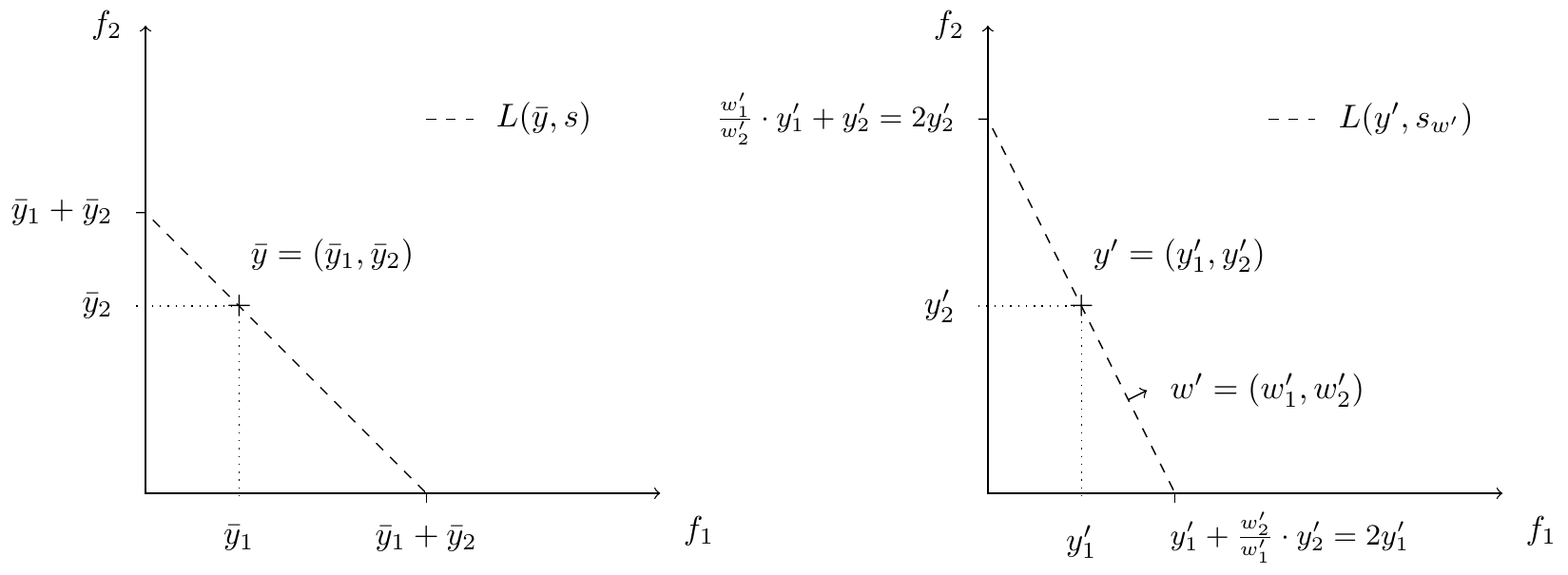}
        \caption{Let $s(y) = y_1 + y_2$ be the (unweighted) sum scalarizing function for $\Pi^{\min} = (\{1,2\},\emptyset)$.
        Left: The component-wise worst case ratio of $\bar{y}$ to any $y^* \in L(\bar{y},s)$ is bounded by \textcolor{black}{$\sup \left\{ \max\left\{\frac{y^*_1}{\bar{y}_1}, \frac{y^*_2}{\bar{y}_2} \right\} \; \bigg\vert \; y^* \in L(\bar{y},s) \right\} = \max \left\{ \frac{\bar{y}_1 + \bar{y}_2}{\bar{y}_1}, \frac{\bar{y}_1 + \bar{y}_2}{\bar{y}_2}\right\} \geq 2.$}
        Right: The component-wise worst case ratio of $y'$ to any $y \in L(y',s_{w'})$, where $w' = \left( \frac{y'_1 + y'_2}{y'_1},  \frac{y'_1 + y'_2}{y'_2} \right)$, is bounded by \textcolor{black}{$\sup \left\{ \max\left\{\frac{y_1}{y'_1}, \frac{y}{y'_2} \right\} \; \bigg\vert \; y \in L(y',s_{w'}) \right\} = \max \left\{ \frac{1}{y'_1} \cdot \left(y'_1 + \frac{w'_2}{w'_1} \cdot y'_2 \right),\frac{1}{y'_2} \cdot \left(\frac{w'_1}{w'_2} \cdot y'_1 +  y'_2\right)  \right\} = \max \left\{ \frac{2 \cdot y'_1}{y'_1}, \frac{2 \cdot y'_2}{y'_2} \right\} = 2.$}}
        \label{fig:level-sets-weighted-sum}
    \end{figure}

\subsection{Tightness Results for Norm-based Weighted Scalarizations}\label{sec:tightness-norm-based-scalarizations}

In the following, we consider scalarizations as in~\eqref{eq:weighted-definition} for which the defining scalarizing function~$s$ is based on norms. We first consider the case that all objective functions are to be minimized and then investigate the case with at least one maximization objective.

\medskip

Note that a norm restricted to the positive orthant is not necessarily a scalarizing function for $\Pi^{\min} = (\{1,\ldots,p\},\emptyset)$.\footnote{For example, consider the norm $\lvert \lvert y \rvert \rvert \coloneqq \lvert y_1 \rvert + \lvert y_2 - y_1 \rvert$ on $\R^2$. Then $(4,2) <_{\Pi^{\min}} (5,5)$, but $\lvert \lvert(4,2)\rvert \rvert = 6 > 5 = \lvert \lvert(5,5)\rvert \rvert$.} Hence, we have to assume that $s$ is strictly $\Pi^{\min}$-monotone. This assumption is satisfied, among others, for all $q$-norms with $1 \leq q \leq \infty$.  The next result states that, for each weighted scalarization induced by~$W = \Rpg$ and a strictly $\Pi^{\min}$-monotone norm~$s$, the computation of the approximation quality given in Theorem~\ref{thm:condition-weighted-scalarization} simplifies to an explicit expression. Moreover, the approximation quality is best possible. 
\cite{Glasser+etal:multi-hardness-proceedings} compute the value of $\alpha$ for the special case of $q$-norms based on constants given by the norm equivalence to the 1-norm. The next result extends this: for each $\Pi^{\min}$-monotone norm, there is actually a closed-form expression for the approximation quality~$\alpha$:
\begin{theorem} \label{thm:norms-minimization}
	Let~$s\colon \R^p \to \R_\geq$ be a strictly $\Pi^{\min}$-monotone norm, let~$S = \{s_w\colon \Rpg \rightarrow \R, s_w(y) = s(w_1 \cdot y_1, \ldots, w_p \cdot y_p) \; \vert \; w \in \Rpg\}$ be the weighted scalarization induced by $W= \Rpg$ and $s$ and denote by~$e^i$ the $i$-th unit vector in~$\R^p$.
    Then, 
	\begin{align*}
	    \inf_{\bar{y} \in \Rpg} \ \sup \left\{\max \left\{ \frac{y^*_1}{\bar{y}_1}, \ldots, \frac{y^*_p}{\bar{y}_p} \right\}\; \bigg\vert \; y^* \in L(\bar{y},s)\right\} = s\left(\frac{1}{s(e^1)},\ldots, \frac{1}{s(e^p)}\right) \eqqcolon \alpha.
	\end{align*}
	Moreover,
	\begin{enumerate}
		\item in each instance of each $p$-objective \textbf{minimization} problem, every optimal solution set for~$S$ is an $\alpha$-approximation set, and
		\item for each $0<\varepsilon<1$, there exists an instance of a $p$-objective \textbf{minimization} problem where the set of $S$-supported solutions is \emph{not} an $\left(\alpha \cdot (1 - \varepsilon) \right)$-approximation set.
	\end{enumerate}
\end{theorem}
\begin{proof}
    For each $\bar{y} \in \Rpg$, Lemma~\ref{lem:sup-norm} in the appendix implies that
	\begin{align*}
		\sup \left\{\max \left\{ \frac{y^*_1}{\bar{y}_1}, \ldots, \frac{y^*_p}{\bar{y}_p} \right\}\; \bigg\vert \; y^* \in L(\bar{y},s)\right\} = s(\bar{y}) \cdot \max\left\{ \frac{1}{s(e^1) \cdot \bar{y}_1},\ldots,\frac{1}{s(e^p) \cdot \bar{y}_p} \right\}.
	\end{align*}
	Then, with ${i_{\min}} \coloneqq \arg \min_{i = 1,,\ldots, p} s(e^{i}) \cdot \bar{y}_{i}$, it holds that
	\begin{align*}
		&s(\bar{y}) \cdot \max\left\{ \frac{1}{s(e^1) \cdot \bar{y}_1},\ldots,\frac{1}{s(e^p) \cdot \bar{y}_p} \right\} \\
		=\ &s\left(\frac{s(e^1) \cdot \bar{y}_1}{s(e^1)},\ldots,\frac{s(e^p) \cdot \bar{y}_p}{s(e^p)} \right) \cdot \frac{1}{s(e^{i_{\min}}) \cdot \bar{y}_{i_{\min}}} \\
		\geq\ &s\left(\frac{s(e^{i_{\min}}) \cdot \bar{y}_{i_{\min}}}{s(e^1)},\ldots,\frac{s(e^{i_{\min}}) \cdot \bar{y}_{i_{\min}}}{s(e^p)} \right) \cdot \frac{1}{s(e^{i_{\min}}) \cdot \bar{y}_{i_{\min}}} \\
		=\  &s\left(\frac{1}{s(e^1)},\ldots, \frac{1}{s(e^p)}\right).
	\end{align*}
	Since choosing $\bar{y} \coloneqq \left( \frac{1}{s(e^1)},\ldots, \frac{1}{s(e^p)} \right) \in \Rpg$ yields 
	\begin{align*}
	s(\bar{y}) \cdot \max\left\{ \frac{1}{s(e^1) \cdot \bar{y}_1},\ldots,\frac{1}{s(e^p) \cdot \bar{y}_p} \right\} =  s\left(\frac{1}{s(e^1)},\ldots, \frac{1}{s(e^p)}\right),
	\end{align*}
	we obtain
	\begin{align*}
    	&\inf_{\bar{y} \in \Rpg} \ \sup \left\{\max \left\{ \frac{y^*_1}{\bar{y}_1}, \ldots, \frac{y^*_k}{\bar{y}_k}, \frac{\bar{y}_{k+1}}{y^*_{k+1}},\ldots,\frac{\bar{y}_{p}}{y^*_{p}} \right\}\; \bigg\vert \; y^* \in L(\bar{y},s)\right\} \\
    	= &\ s\left(\frac{1}{s(e^1)},\ldots, \frac{1}{s(e^p)}\right).
	\end{align*}
   Then, Statement~$1$ follows by Theorem~\ref{thm:condition-weighted-scalarization} since the infimum is in fact attained. 

    \medskip
    
    We now prove Statement~$2$. Assume without loss of generality that~$s(e^i) = 1$ for~$i = 1,\ldots,p$. Otherwise use the norm $s' \coloneqq s_w$ for $w \in \Rpg$ with $w_i = \frac 1 {s(e^i)}$ instead. Then $S = \{s'_w \; \vert \; w \in \Rpg\}$, $s'(e^i) = s(w_i \cdot e^i) = w_i \cdot s(e^i) = 1$, and $\alpha = s\left(\frac 1 {s(e^1)}, \ldots, \frac 1 {s(e^p)}\right) = s'(1,\ldots, 1) = s'\left(\frac 1 {s'(e^1)}, \ldots, \frac 1 {s'(e^p)}\right)$.
    
    Given~$0<\varepsilon <1$, first define vectors~$\tilde{e}^1,\ldots,\tilde{e}^p \in \Rpg$ with~$\tilde{e}^i_j \coloneqq e^i_j + \delta$ for~$i,j \in\{1,\ldots,p\}$, where $\delta \coloneqq \frac{\varepsilon}{2 \alpha}$. Then, define a $p$-objective minimization problem instance $(X,f)$ with $X = \{\bar x, x^{(1)},\ldots, x^{(p)}\}$ via
    \begin{align*}
    f_j(\bar x) \coloneqq \frac 1 \alpha + \delta \textnormal{ for } j = 1,\ldots, p
    \end{align*}
    and
    \begin{align*}
        f(x^{(i)}) \coloneqq \left(1 - \frac \varepsilon 2\right) \cdot \tilde{e}^i \textnormal{ for } i = 1,\ldots, p.
    \end{align*}
    Then the solution~$\bar x$ is not $\left(\alpha \cdot (1 - \varepsilon)\right)$-approximated by any other solution~$x^{(i)}$: For each $i = 1,\ldots,p$, we have $f_i(\bar x) = \frac{ 1 + \frac\varepsilon 2} \alpha$ and $f_i(x^{(i)}) = \left(1 - \frac \varepsilon 2\right) \cdot \tilde{e}^i_i = \left(1 - \frac \varepsilon 2\right) \cdot (1+\delta)$ and, thus,
    \begin{align*}
        (1-\varepsilon) \cdot \alpha \cdot f_i(x^{(i)})= (1-\varepsilon) \cdot \left(1 + \frac \varepsilon 2\right) < 1 - \frac \varepsilon 2 < f_i(\bar{x}).
    \end{align*}
    Moreover, for each $w \in \Rpg$, the solution~$\bar x$ is not optimal for~$s_w$: Given $w \in \Rpg$, choose $i \in \{1,\ldots,p\}$ such that $w_i = \min_{j = 1,\ldots, p} w_j$. Then
    \begin{align*}
        s_w(\tilde{e}^i) \leq s_w(e^i) + \delta \cdot s_w(1,\ldots, 1)
        = w_i \cdot s(e^i) + \delta \cdot s(w)
        = w_i + \delta \cdot s(w),
    \end{align*}
    where the inequality follows by the triangle inequality.
    This implies that
    \begin{align*}
        s_w(f(x^{(i)})) &= \left(1-\frac \varepsilon 2 \right) \cdot s_w(\tilde{e}^i)\\
        &< s_w(\tilde{e}^i)\\
        &\leq w_i + \delta \cdot s(w)\\
        &= \frac 1 \alpha \cdot s(w_i,\ldots, w_i) + \delta \cdot s(w)\\
        &\leq \left(\frac 1 \alpha + \delta\right) \cdot s(w)\\
        &= s_w(f(\bar x)),
    \end{align*}
    which concludes the proof.
\end{proof}
Table~\ref{tbl:approximation-qualities-weighted-scalarizations} presents the approximation qualities given by Theorem~\ref{thm:norms-minimization} of the most frequently used norms in the context of multiobjective optimization.
\begin{table}[b]
     \centering
     \begin{tabular}{l|c|c|c|c}
       norm $s(y)$ & $\sum\limits_{i=1}^p y_i$ & $\left(\sum\limits_{i=1}^p y^q_i\right)^{\frac{1}{q}}$ & $ \max\limits_{i=1,\ldots,p} y_i$ & $\sum\limits_{i=1}^p y_i + \varrho \cdot \max\limits_{i=1,\ldots,p} \{ y_i\}$  \\
       \hline
       approx.\ qual.~$\alpha$ & $p$ & $p^{\frac{1}{q}}$ & 1 & $\frac{p + \varrho}{1+\varrho}$
     \end{tabular}

     \caption{Approximation qualities guaranteed by Theorem~\ref{thm:norms-minimization} for weighted scalarizations implied by the $1$-norm, a $q$-norm with $1 \leq q < \infty$, the Tchebycheff norm, and the modified augmented  Tchebycheff norm with $\varrho > 0$. In each case, the chosen reference point is the origin.}
    \label{tbl:approximation-qualities-weighted-scalarizations}
\end{table}

\medskip 

Theorem~\ref{thm:norms-minimization} can be further generalized: Recall Remark~\ref{rem:dual-alternatives} which illustrates alternative definitions for $\Gamma$-transformed scalarizing functions by means of continuous strictly increasing functions. This is possible since, given a scalarizing function~$s$, the optimal solution sets of the induced single-objective optimization problem instance do not change under the concatenation of any continuous strictly increasing function~$g$ with~$s$. The same reasoning also yields that the class of scalarizations, for which the approximation quality given in Theorem~\ref{cor:sup-inf-sup} can be stated explicitly and is best possible, is even broader:
\begin{corollary}
    Let $s\colon\R^p \rightarrow \R_\geq$ be a strictly $\Pi^{\min}$-monotone norm and let $g\colon(\Rpg) \rightarrow \R$ be a continuous strictly increasing function. Let $\tilde{S} = \{\tilde{s}_w\colon \Rpg \rightarrow \R, \tilde{s}_w(y) = g(s(w_1 \cdot y_1,\ldots, w_p \cdot y_p)) \, \vert \, w \in \Rpg\}$ be the weighted scalarization for minimization induced by $W = \Rpg$ and the concatenation of $g$ and $s$. Further, let $\alpha \coloneqq s\left(\frac{1}{s(e^1)}, \ldots, \frac{1}{s(e^p)} \right)$, where $e^i$ denotes the $i$-th unit vector in $\R^p$. Then,
    \begin{enumerate}
		\item in each instance of each $p$-objective \textbf{minimization} problem, every optimal solution set for~$\tilde{S}$ is an $\alpha$-approximation set, and
		\item for each $0<\varepsilon<1$, there exists an instance of a $p$-objective \textbf{minimization} problem where the set of $\tilde{S}$-supported solutions is \emph{not} an $\left(\alpha \cdot (1 - \varepsilon) \right)$-approximation set.
	\end{enumerate}
\end{corollary}
\begin{proof}
     Note that, since $g\colon (\Rpg) \rightarrow \R$ is continuous and strictly increasing, $\tilde{s}$ is indeed a scalarizing function for $\Pi^{\min} = (\{1,\ldots,p\},\emptyset)$. In particular, for each $\bar{y},y^*\in \Rpg$, it holds that $g(s(\bar{y})) = g(s(y^*))$ if and only if $s(\bar{y}) = s(y^*)$ and, therefore, $L(\bar{y},\tilde{s}) =  L(\bar{y}, s)$. Hence, the claim follows immediately from Theorem~\ref{thm:norms-minimization}.
\end{proof}

Next, we consider the case that at least one maximization objective is given. Again, let $\Pi = (\{1,\ldots,k\},\{k+1,\ldots,p\})$ for some $0\leq k < p$ be given without loss of generality.
Besides the transformation presented in Section~\ref{sec:duality}, another adaption of strictly $\Pi^{\min}$-monotone norms to scalarizing functions for~$\Pi$ is to first combine all minimization objectives by means of the norm projected to the first $k$-objectives, combine all maximization objectives by means of the norm projected to the last $p-k$-objectives, and subtract the norm value of the maximization objectives from the norm value of the minimization objective. If applied to the 1-norm, we obtain in such a way the different weighted sum scalarizing functions introduced in Example~\ref{ex:weighted_sum}. A formal and even more general definition is given in the next lemma:
\begin{lemma}\label{lem:general-norm-based-scalarization}
    Let $\Pi = (\{1,\ldots,k\},\{k+1,\ldots, p\})$. Let $s^1\colon\R^k \rightarrow \R_\geq$ be a strictly $(\{1,\ldots,k\},\emptyset)$-monotone norm on $\R^k$, and let $s^2\colon\R^{p-k} \rightarrow \R_\geq$ be a strictly $(\{1,\ldots,p-k\},\emptyset)$-monotone norm on $\R^{p-k}$. Define the function~$s\colon\Rpg \rightarrow \R,\ s(y) \coloneqq s^1(y_1,\ldots, y_k) - s^2(y_{k+1},\ldots,y_p)$. Then, $s$ is a scalarizing function for $\Pi$. 
\end{lemma}
\begin{proof}
     The function~$s$ is continuous since $s^1$ as well as $s^2$ are continuous. Let $y,y' \in \Rpg$ such that $y <_{\Pi} y'$. Then, $y_i < y'_i$ for all $i = 1, \ldots, k$ and $y_i > y'_i$ for all $i = k+1, \ldots, p$. Since $s^1$ is strictly $(\{1,\ldots,k\},\emptyset)$-monotone and $s^2$ is strictly $(\{1,\ldots,p-k\},\emptyset)$-monotone, it holds that $s^1(y_1,\ldots,y_k) < s^2(y'_1,\ldots,y'_k)$ and $s^1(y_{k+1},\ldots,y_p) > s^2(y'_{k+1},\ldots,y'_p)$, and, therefore, $s(y) < s(y')$.
\end{proof}
The next result states that $S$-supported solutions, where $S$ is a  weighted scalarization induced by $W = \Rpg$ and a scalarizing function $s$ as in Lemma~\ref{lem:general-norm-based-scalarization}, are no approximation set in general. In particular, this generalizes the impossibility results concerning the approximation of multiobjective maximization problems via the weighted sum scalarization presented in~\cite{Bazgan+etal:power-weighted-sum,Glasser+etal:multi-hardness-proceedings,Glasser+etal:multi-hardness,Halfmann+etal:general-approx,Helfrich+etal:cones}. 
\begin{theorem}\label{thm:norm-based-scalarization-with-max}
	Let $\Pi = (\{1,\ldots,k\},\{k+1,\ldots, p\})$ such that $0 \leq k < p$. Let a scalarizing function~$s$ for $\Pi$ be given as in Lemma~\ref{lem:general-norm-based-scalarization} and let~$S = \{ s_w\colon\Rpg \rightarrow \R, s_w(y) \coloneqq s(w_1 \cdot y_1, \ldots, w_p \cdot y_p) \; \vert \, w \in \Rpg \}$ be the weighted scalarization induced by $W = \Rpg$ and $s$. Then,
	\begin{enumerate}
		\item For any $\alpha \geq 1$, there exists an instance of a $p$-objective optimization problem of type~$\Pi$ such that the set of $S$-supported solutions is not an $\alpha$-approximation set.
		\item It holds that $$\inf_{\bar{y} \in \Rpg} \ \sup \left\{\max \left\{ \frac{y^*_1}{\bar{y}_1}, \ldots, \frac{y^*_k}{\bar{y}_k}, \frac{\bar{y}_{k+1}}{y^*_{k+1}},\ldots,\frac{\bar{y}_{p}}{y^*_{p}} \right\}\; \bigg\vert \; y^* \in L(\bar{y},s)\right\} = \infty.$$
	\end{enumerate}
\end{theorem}
\begin{proof} 
    Statement~2 follows immediately by Statement~1 and Theorem~\ref{thm:condition-weighted-scalarization}. Hence, it is left to prove Statement~1.
    
    \medskip
    
    For the sake of simplification, we denote, for any point~$y \in \Rpg$, by $s^1(y)$ the application of $s^1$ to the projection of $y$ to the components $1,\ldots,k$. Similarly, we denote by $s^2(y)$ the application of $s^2$ to the projection of $y$ to the components $k+1,\ldots,p$.
     
    Let~$e^i$ denote the $i$-th unit vector in $\R^p$ and assume without loss of generality that $s^1(e^i) = 1$ for $i = 1,\ldots, k$ and $s^2(e^i)= 1$ for $i = k+1,\ldots,p$. Otherwise, use the function~$s'\colon \Rpg \rightarrow \R, s'(y) = s^1 (w_1 \cdot y_1,\ldots, w_p \cdot y_p)- s^2 (w_{k+1} \cdot y_{k+1},\ldots, w_p \cdot y_p)$ with $w = \left(\frac{1}{s^1(e^1)},\ldots, \frac{1}{s^1(e^k)},\frac{1}{s^2(e^{k+1})},\ldots, \frac{1}{s^2(e^p)}\right)$ instead. Then, $s^1(w_i \cdot e^i) = w_i \cdot s^1(e^i) = 1$ for $i = 1,\ldots k$ and $s^2(w_i \cdot e^i) = w_i \cdot s^2( e^i) = 1$ for $i = k+1,\ldots,p$. Additionally, $S = \{s'_w \; \vert \; w \in \Rpg\}$ and
	\begin{align*}
		&\inf_{\bar{y} \in \Rpg} \ \sup \left\{\max \left\{ \frac{y^*_1}{\bar{y}_1}, \ldots, \frac{y^*_k}{\bar{y}_k}, \frac{\bar{y}_{k+1}}{y^*_{k+1}},\ldots,\frac{\bar{y}_{p}}{y^*_{p}} \right\}\; \Bigg\vert \; y^* \in L(\bar{y},s)\right\} \\
		= &\inf_{\bar{y} \in \Rpg} \ \sup \left\{\max \left\{ \frac{y^*_1}{\bar{y}_1}, \ldots, \frac{y^*_k}{\bar{y}_k}, \frac{\bar{y}_{k+1}}{y^*_{k+1}},\ldots,\frac{\bar{y}_{p}}{y^*_{p}} \right\}\; \Bigg\vert \; y^* \in L(\bar{y},s')\right\},
	\end{align*} 
	see the proof of Lemma~\ref{lem:sup-inf-sup-for-weighted}.
	
	In order to prove Statement~1, we distinguish whether there exists exactly one objective function to be maximized ($k = p-1$) or at least two objective functions to be maximized ($k < p-1$).
     
    We first prove the case that $k = p-1$. Given $\alpha \geq 1$, choose $m \in \R$ such that $0< m < 1 - \frac{\alpha}{\alpha + 1}$, and choose $M \in \R$ such that	$M > \alpha  + 1 \geq 1.$
	Then, define an instance of a $p$-objective optimization problem of type~$\Pi$ with $X = \{\bar{x},x^{(1)},x^{(2)}\}$ via
	$f(\bar{x}) \coloneqq (1,\ldots,1)$ and $f(x^{(1)}) \coloneqq (m,\ldots, m, \frac{1}{\alpha + 1})$ and $f(x^{(2)}) \coloneqq (\alpha + 1,\ldots, \alpha + 1,M)$. 
	Then, $\bar{x}$ is not $\alpha$-approximated by $x^{(1)}$ since
	\begin{align*}
		\alpha \cdot f_p(x^{(1)}) = \frac{\alpha}{\alpha + 1} < 1 = f_p(\bar{x}),
	\end{align*}
	and $\bar{x}$ is not $\alpha$-approximated by $x^{(2)}$ since
	\begin{align*}
	    \alpha \cdot f_1(\bar{x}) = \alpha < \alpha + 1 = f_1(x^{(2)}).
	\end{align*}
	Moreover, for each $w \in \Rpg$, the solution~$\bar{x}$ is not optimal for $s_w$: Let~$w \in \Rpg$ be given. If $s^1(w_1,\ldots,w_{p-1}) \geq s^2(w_p)$, it holds that
	\begin{align*}
    	s_w(f(x^{(1)})) &= s^1 (w_1 \cdot m, \ldots, w_{p-1} \cdot m) - s^2(w_p \cdot \frac{1}{\alpha + 1})\\
    	&= m \cdot s^1(w_1,\ldots,w_{p-1}) - \frac{1}{\alpha +1} s^2(w_p) \\
    	&= m \cdot s^1(w_1,\ldots,w_{p-1}) - \frac{\alpha + 1 - \alpha}{\alpha +1} s^2(w_p) \\
    	&= s^1(w_1,\ldots,w_{p-1}) - s^2(w_p) + (m-1)  s^1(w_1,\ldots,w_{p-1}) + \frac{\alpha}{\alpha +1} s^2(w_p)\\
    	&\leq s_w(f(\bar{x})) + (m-1)  s^1(w_1,\ldots,w_{p-1}) + \frac{\alpha}{\alpha +1} s^1(w_1,\ldots,w_{p-1}) \\
    	&= s_w(f(\bar{x})) + \left(m-1 + \frac{\alpha}{\alpha +1}\right) s^1(w_1,\ldots,w_{p-1})\\
    	&< s_w(f(\bar{x})) + \left(1 - \frac{\alpha}{\alpha +1} -1 + \frac{\alpha}{\alpha +1} \right) s^1(w_1,\ldots,w_{p-1}) = s_w(f(\bar{x})).
	\end{align*}
	Otherwise, if $s^1(w_1,\ldots,w_{p-1}) \leq s^2(w_p)$, it holds that
	\begin{align*}
		s_w(f(x^{(2)})) &=  s^1(w_1 \cdot (\alpha + 1),\ldots,w_{p-1} \cdot (\alpha +1) ) -  s^2(w_p \cdot M)\\
		&= (\alpha + 1) \cdot s^1(w_1,\ldots,w_{p-1}) - M \cdot s^2(w_p) \\
		&= s^1(w_1,\ldots,w_{p-1}) - s^2(w_p) + \alpha \cdot s^1(w_1,\ldots,w_{p-1}) - (M-1)  s^2(w_p) \\
		&\leq s_w(f(\bar{x})) + (\alpha - M + 1) \cdot s^2(w_p) \\
		&< s_w(f(\bar{x})) + (\alpha - (\alpha + 1) + 1) \cdot s^2(w_p) = s_w(f(\bar{x})).
	\end{align*}
	Hence, the case $k = p-1$ is proven. 
	
	Now, let $k < p-1$. Given $\alpha \geq 1$, we define an instance of a $p$-objective optimization problem of type~$\Pi$ with $X = \{\bar{x}, x^{(k+1)},\ldots, x^{(p)}\}$ via $f(\bar{x}) \coloneqq (1,\ldots,1)$ and, for $j = k+1,\ldots,p$, 
	\begin{align*}
		f_i(x^{(j)}) &\coloneqq \frac{1}{2},\quad \quad \quad \quad \quad \text{ if }  i = 1,\ldots, k,\\
		f_j(x^{(j)}) &\coloneqq s^2(1,\ldots, 1),\\
		f_i(x^{(j)}) &\coloneqq \frac{1}{\alpha + 1}, \quad \quad \quad \text{ if } i = k+1,\ldots, p,\ i \neq j. 
	\end{align*} 
	Then, $\bar{x}$ is not $\alpha$-approximationed by $x^{(j)}$, $j \in \{k+1,\ldots, p\}$, since there is an $i \in  \{k+1,\ldots, p\} \setminus \{j\}$ such that
	\begin{align*}
	    \alpha \cdot f_i(x^{(i)}) = \frac{\alpha}{\alpha + 1} < 1  = f_i(\bar{x}).
	\end{align*}
	Moreover, for each $w \in \Rpg$, the solution~$\bar{x}$ is not optimal for $s_w$: For~$w \in \Rpg$, let $j \coloneqq \arg \max_{i=k+1,\ldots,p} w_i$. Then,
	\begin{align*}
		s_w(f(x^{(j)})) &= s^1(w_1 f_1(x^{(j)}),\ldots,w_k f_k(x^{(j)})) - s^2( w_{k+1}  f_{k+1}(x^{(j)}),\ldots, w_p  f_p(x^{(j)})) \\
		&< s^1(w_1,\ldots,w_k) - s^2(w_j \cdot s^2(1,\ldots,1) \cdot e^j) \\
		&= s^1(w_1,\ldots,w_k) - w_j \cdot s^2(1,\ldots,1) \cdot s^2(e^j) \\
		&= s^1(w_1,\ldots,w_k) - s^2(w_j \ldots,w_j) \\
		&\leq  s^1(w_1,\ldots,w_k) - s^2(w_{k+1} \ldots,w_p) = s_w(f(\bar{x})),
	\end{align*}
	where the inequalities follows since $s^1$ is strictly $(\{1,\ldots,k\},\emptyset)$-monotone and the application of Lemma~\ref{lem:weak-monotonicity} to $s^2$. This concludes the proof.
\end{proof}
Analogously to the minimization case, a concatenation of any continuous strictly increasing function with the scalarizing function does no affect the result of Theorem~\ref{thm:norm-based-scalarization-with-max}:
\begin{corollary}
    Let $\Pi = (\{1,\ldots,k\},\{k+1,\ldots, p\})$ such that $0 \leq k < p$. Let a scalarizing function~$s$ for $\Pi$ be given as in Lemma~\ref{lem:general-norm-based-scalarization} and let $g\colon (\Rpg) \rightarrow \R$ be a continuous strictly increasing function. Let $\tilde{S} = \{\tilde{s}_w\colon \Rpg \rightarrow \R, \tilde{s}_w(y) = g(s(w_1 \cdot y_1,\ldots, w_p \cdot y_p)) \, \vert \, w \in \Rpg\}$ be the weighted scalarization for minimization induced by $W = \Rpg$ and the concatenation of $g$ and $s$. Then, for any $\alpha \geq 1$, there exists an instance of a $p$-objective optimization problem of type~$\Pi$ such that the set of $\tilde{S}$-supported solutions is not an $\alpha$-approximation set.
\end{corollary}

\section{Discussion and Conclusion}
Until now, scalarizations that yield  an approximation set in each instance are only known for the case of pure multiobjective minimization problems. In fact, concerning all scalarizations for maximization studied so far in the context of approximation, only impossibility results are known, and we are not aware of any work that studies the approximation via scalarizations for the case that both minimization and maximization objectives are present.

In this work, we establish that, from a theoretical point of view, all optimization problems can be approximated equally well via scalarizations. In particular, for each objective decomposition, scalarizations can be constructed that yield the same approximation quality. 
This is possible due the existence of powerful scalarizations for the approximation of multiobjective minimization problems such as the weighted sum scalarization, see~Example~\ref{ex:weighted_sum}, or  norm-based weighted scalarizations, see~Theorem~\ref{thm:norms-minimization}: for each instance of each multiobjective minimization problem, every optimal solution set yields an approximation quality that depends solely on the scalarization itself. Our results of Section~\ref{sec:duality} show that the above scalarizations can, for each other decomposition~$\Pi$, appropriately be transformed such that the same holds true: in each instance of each multiobjective minimization problem of type~$\Pi$, every optimal solution set for the transformed scalarization yields an approximation quality meeting exactly the approximation quality given by the original scalarization.

It should be noted that the scalarizing functions of the transformations of the above mentioned scalarizations turn out to be nonlinear. Therefore, the associated instances of the single-objective optimization problems are surmised to be difficult to solve exactly in general, even when using heuristics or programming methods that sacrifice polynomial-time running time. Hence, follow up research is motivated: do scalarizations for objective decompositions including maximization objectives exist that yield an a priori identifiable approximation quality in arbitrary instances and whose implied single-objective problem instances are solvable from a theoretical and/or practical point of view?
Theorem~\ref{thm:norm-based-scalarization-with-max} rules out the majority of scalarizations studied and applied until now in the context of multiobjective optimization. Nevertheless, the findings of Section~\ref{sec:conditions} indicate guidelines on conditions for the scalarizing functions of a potential scalarization. 

Another crucial question relates to the tightness of the upper bound on the best approximation quality given in Theorem~\ref{cor:sup-inf-sup}. Example~\ref{ex:approx-not-tight-in-general} shows that, in general, the upper bound is not tight. However, for the majority of norm-based scalarizations applied in the context of multiobjective optimization, the upper bound is in fact best possible, see Section~\ref{sec:tightness-norm-based-scalarizations}. What conditions on scalarizations imply that the given approximation quality is best possible? Do general weighted scalarizations meet these conditions?

A third direction of research could be a study of scalarization in view of a component-wise approximation as, for example, considered in~\cite{Bazgan+etal:power-weighted-sum,Herzel+etal:dualrestrict,Halfmann+etal:general-approx}. Hereby, we note that the results of Section~\ref{sec:duality}, Theorem~\ref{thm:finite-scalarizations} and Lemma~\ref{lem:two-solutions-sufficient},  are easy to generalize to this case. 
However, to obtain necessary conditions for a scalarization for (component-wise) approximation similar to the results of Proposition~\ref{prop:scalarizationset-necessary}, the infimum operator must be replaced by a concept for vectors of approximation qualities in order to specify what \enquote{the best approximation factors} means. Hence, the study of scalarizations in view of a component-wise approximation can potentially be connected to the multi-factor notion of approximation introduced in~\cite{Bazgan+etal:power-weighted-sum}.

\section*{Acknowledgments}
This work was funded by the Deutsche Forschungsgemeinschaft (DFG, German Research Foundation) – Project number~398572517.




\begin{appendices}
\section{A Bound on Ratios in Level Sets}
\begin{lemma}\label{lem:sup-norm}
    Let $s\colon \mathbb{R}^p \rightarrow \R$ be a strictly $(\{1,\ldots,p\},\emptyset)$-monotone norm. For each $\bar{y} \in \Rpg$, it holds that
	\begin{align*}
		\sup \left\{\max \left\{ \frac{y^*_1}{\bar{y}_1}, \ldots, \frac{y^*_p}{\bar{y}_p} \right\}\; \bigg\vert \; y^* \in L(\bar{y},s)\right\} = \max\left\{ \frac{s(\bar{y})}{s(e^1) \cdot \bar{y}_1},\ldots,\frac{s(\bar{y})}{s(e^p) \cdot \bar{y}_p} \right\},
	\end{align*}
	where $e^i$ denotes the $i$-th unit vector in~$\mathbb{R}^p$.
\end{lemma}
\begin{proof}
	For each $y^* \in L(\bar{y},s)$, by the triangle inequality, the nonnegativity and monotonicity of the norm, and $y^* \in \Rpg$, it holds that
	\begin{align*}
    	s(\bar{y}) = s(y^*) = s\left(\sum_{j=1}^p y^*_j \cdot e^{j}\right) \geq s(y^*_i \cdot e^{i}) = y^*_i \cdot s(e^i) \text{ for all } i = 1,\ldots, p,
	\end{align*}
    which implies that
	\begin{align*}
    	\max\left\{ \frac{y^*_1}{\bar{y}_1}, \ldots, \frac{y^*_p}{\bar{y}_p}\right\} \leq \max\left\{ \frac{s(\bar{y})}{s(e^1) \cdot \bar{y}_1},\ldots,\frac{s(\bar{y})}{s(e^p) \cdot \bar{y}_p} \right\}.
	\end{align*}
	This shows that the supremum on the left-hand side in the claim is less than or equal to the term on the right-hand side. In order to show that equality holds,
	we choose $i_{\max} = {\arg \max}_{i=1,\ldots,p} \frac{s(\bar{y})}{s(e^i) \cdot \bar{y}_i}$ and construct a sequence~$\left( y^{(n)}\right)_{n \in \mathbb{N}} \subseteq L(\bar{y},s)$ such that 
    	\begin{align*}\lim_{n \rightarrow \infty}\max\left\{ \frac{y^{(n)}_1}{\bar{y}_1}, \ldots, \frac{y^{(n)}_p}{\bar{y}_p}\right\}  =  \frac{s(\bar{y})}{s(e^{i_{\max}}) \cdot \bar{y}_{i_{\max}}}.
	\end{align*}
    This is done by initially constructing a sequence in $\Rpg$ converging to $\frac{s(\bar{y})}{s(e^{i_{\max}})} \cdot e^{i_{\max}}$ which is then projected on the level set $L(\bar{y},s)$ by appropriately chosen scaling factors. Since $\frac{s(\bar{y})}{s(e^{i_{\max}})} \cdot e^{i_{\max}}$ is contained in the closure of $L(\bar{y},s)$, the projected sequence also converges to  $\frac{s(\bar{y})}{s(e^{i_{\max}})} \cdot e^{i_{\max}}$:
    
    For each $n \in \mathbb{N}$, define a vector~$\tilde{y}^{(n)} \in \Rpg$ by 
	\begin{align*}
    	\tilde{y}^{(n)}_{i_{\max}} \coloneqq \frac{s(\bar{y})}{s(e^{{i_{\max}}})}, \ \tilde{y}^{(n)}_j \coloneqq \frac{1}{n}, j = 1, \ldots, p, j \neq {i_{\max}}.
	\end{align*}		
	Then, for each $n \in \mathbb{N}$, it holds that $\tilde{y}^{(n)} \geq \frac{s(\bar{y})}{s(e^{{i_{\max}}})} \cdot e^{{i_{\max}}}$, which implies that $s(\tilde{y}^{(n)}) \geq s(\frac{s(\bar{y})}{s(e^{{i_{\max}}})} \cdot e^{{i_{\max}}}) = s(\bar{y})$. If $s(\tilde{y}^{(n)}) = s(\bar{y})$, set $\lambda_n \coloneqq 1$. 
	In the case that $s(\tilde{y}^{(n)}) > s(\bar{y})$, Lemma~\ref{lem:beam} implies that there exists a scalar~$0 < \lambda_n < 1 $ such that $s(\lambda_n \cdot \tilde{y}^{(n)}) = s(\bar{y})$. 
	Since, for each $n \in \mathbb{N}$, it holds that $s(\bar{y}) = s(\lambda_n \cdot \tilde{y}^{(n)}) = \lambda_n \cdot s(\tilde{y}^{(n)})$ and since $s$ is continuous, we obtain
	\begin{align*}
		\lim_{n \rightarrow \infty} \lambda_n = \lim_{n \rightarrow \infty} \frac{s(\bar{y})}{s(\tilde{y}^{(n)})} = \frac{s(\bar{y})}{s \left(\lim_{n \rightarrow \infty} \tilde{y}^{(n)}\right)} = \frac{s(\bar{y})}{s \left( \frac{s(\bar{y})}{s(e^{i_{\max}})} \cdot e^{i_{\max}}\right)} = 1.
	\end{align*}	
	For each $n \in \mathbb{N}$, define the vector~$y^{(n)} \coloneqq \lambda_n \cdot \tilde{y}^{(n)}$. Then, $\left( y^{(n)}\right)_{n \in \mathbb{N}} \subseteq L(\bar{y},s)$ by choice of~$\lambda_n$ and, since $\lim_{n \rightarrow \infty} \lambda_n = 1$ and $\lim_{n \rightarrow \infty} \tilde{y}^{(n)} = \frac{s(\bar{y})}{s(e^{{i_{\max}}})} \cdot e^{i_{\max}}$, it additionally holds that $
	\lim_{n \rightarrow \infty} y^{(n)} = \frac{s(\bar{y})}{s(e^{{i_{\max}}})} \cdot e^{i_{\max}}.
	$
	Consequently,
	\begin{align*}
		\lim_{n \rightarrow \infty} \max\left\{ \frac{y^{(n)}_1}{\bar{y}_1}, \ldots, \frac{y^{(n)}_p}{\bar{y}_p}\right\} 
		= \lim_{n \rightarrow \infty} \frac{y^{(n)}_{i_{\max}}}{\bar{y}_{i_{\max}}}  = \frac{s(\bar{y})}{s(e^{i_{\max}}) \cdot \bar{y}_{i_{\max}}}
	\end{align*}
	and the claim is proven.
\end{proof}
\end{appendices}

\end{document}